\documentclass[reqno,a4paper,12pt]{amsart}
\usepackage[utf8]{inputenc}
\usepackage[T1]{fontenc}
\usepackage[english]{babel}
\usepackage[babel]{microtype}
\usepackage[a4paper, margin=1.1in]{geometry}
\usepackage[dvipsnames]{xcolor}
\usepackage{csquotes}
\usepackage{amsmath,amssymb,amsthm}
\usepackage{mathrsfs}
\usepackage{mathtools}
\usepackage{commath}
\usepackage{thmtools}
\usepackage{thm-restate}
\usepackage{bbm}
\usepackage{dsfont}
\usepackage{lmodern}
\usepackage{fixcmex}

\usepackage{cases}
\usepackage{enumitem}
\setlist[enumerate,1]{label=(\roman*)}
\usepackage{centernot}

\usepackage[pdfborderstyle={/S/U/W 0},unicode]{hyperref}
\hypersetup{
    colorlinks=true,
    linkcolor=blue!85!black,
    citecolor=blue!85!black,
    urlcolor=blue!85!black,
}

\usepackage{cleveref}
\usepackage{etoolbox}
\numberwithin{equation}{section}
\ifdefined\thmcolor
\declaretheoremstyle[
  shaded={bgcolor=\thmcolor,
  }
]{plain}
\else
\fi
\ifdefined\defcolor
\declaretheoremstyle[
  headfont=\normalfont\bfseries,
  bodyfont=\normalfont,
  shaded={bgcolor=\defcolor}
]{noital}
\else
\declaretheoremstyle[
  headfont=\normalfont\bfseries,
  bodyfont=\normalfont,
]{noital}
\fi
\declaretheorem[style=plain,numberwithin=section,name=Theorem]{theorem}
\declaretheorem[style=plain,sibling=theorem,name=Proposition]{proposition}
\declaretheorem[style=plain,sibling=theorem,name=Lemma]{lemma}

\declaretheorem[style=plain,sibling=theorem,name=Conjecture]{conjecture}
\declaretheorem[style=plain,sibling=theorem,name=Question]{question}
\declaretheorem[style=noital,sibling=theorem,name=Remark]{remark}

\newcommand{\indef}[1]{\emph{#1}}
\newcommand{\defined}{\mathrel{\coloneqq}}
\DeclarePairedDelimiter{\p}{\lparen}{\rparen}
\renewcommand{\leq}{\leqslant}
\renewcommand{\geq}{\geqslant}
\newcommand{\st}{\mathbin{\colon}}
\undef{\set}
\DeclarePairedDelimiter{\set}{\lbrace}{\rbrace}
\undef{\emptyset}
\newcommand{\emptyset}{\varnothing}
\DeclarePairedDelimiter{\card}{\lvert}{\rvert}
\newcommand{\from}{\colon}
\DeclarePairedDelimiter{\floor}{\lfloor}{\rfloor}
\DeclarePairedDelimiter{\ceil}{\lceil}{\rceil}
\DeclareMathOperator{\loglog}{log\,log}
\newcommand{\divides}{\mathrel{\mid}}
\newcommand{\ndivides}{\mathrel{\nmid}}
\undef{\mod}
\newcommand{\mod}[1]{\ (\mathrm{mod}\ #1)}
\DeclareMathOperator{\lcm}{lcm}
\newcommand{\eps}{\varepsilon}
\newcommand{\NN}{\mathbb{N}}
\newcommand{\RR}{\mathbb{R}}
\newcommand{\ZZ}{\mathbb{Z}}
\newcommand{\cA}{\mathcal{A}}
\newcommand{\cB}{\mathcal{B}}
\newcommand{\cC}{\mathcal{C}}

\newcommand{\cN}{\mathcal{N}}
\newcommand{\cV}{\mathcal{V}}
\newcommand{\rB}{\mathrm{B}}
\newcommand{\rG}{\mathrm{G}}
\newcommand{\rP}{\mathrm{P}}
\newcommand{\rR}{\mathrm{R}}
\newcommand{\bfa}{\mathbf{a}}
\newcommand{\bfs}{\mathbf{s}}
\newcommand{\bft}{\mathbf{t}}
\newcommand{\bfu}{\mathbf{u}}
\newcommand{\bfv}{\mathbf{v}}
\newcommand{\bfw}{\mathbf{w}}
\newcommand{\bfx}{\mathbf{x}}
\newcommand{\bfy}{\mathbf{y}}
\newcommand{\bfz}{\mathbf{z}}

\usepackage{tikz}
\usetikzlibrary{arrows, arrows.meta}

\newcommand{\ccint}[3]{
\draw [{Bracket[scale=2]}-{Bracket[scale=2]}, thick, #3] (#1,0) -- (#2,0);
\fill[opacity = 1, #3]
(#1, -0.010) -- (#1, 0.010)
-- (#2, 0.010) -- (#2, -0.010)
-- cycle;
}

\newcommand{\ocint}[3]{
\draw [{Parenthesis[scale=2]}-{Bracket[scale=2]}, thick, #3] (#1,0) -- (#2,0);
\fill[opacity = 1, #3]
{[rounded corners=2.5] (#2, -0.010) -- (#1, -0.010) -- (#1, 0.010)}
-- (#2, 0.010) -- cycle;
}

\newcommand{\coint}[3]{
\draw [{Bracket[scale=2]}-{Parenthesis[scale=2]}, thick, #3] (#1,0) -- (#2,0);
\fill[opacity = 1, #3]
{[rounded corners=2.5] (#1, 0.010)
-- (#2, 0.010) -- (#2, -0.010)}
-- (#1, -0.010) -- cycle;
}

\usepackage[backend=biber, style=ext-numeric, articlein=false, maxnames=99, giveninits=true, doi=false, isbn=false, url=false]{biblatex}
\bibliography{monochromatic-products-arxiv-v2.bib}

\begin{document}

\title[The number of monochromatic solutions to multiplicative equations]{On the number of monochromatic solutions to multiplicative equations}

\author[L. Aragão]{Lucas Aragão}
\address{IMPA, Estrada Dona Castorina 110, Jardim Botânico, Rio de Janeiro, RJ, Brasil}
\email{l.aragao@impa.br}

\author[J. Chapman]{Jonathan Chapman}
\address{School of Mathematics, University of Bristol, Bristol, BS8 1UG, UK, and the Heilbronn
Institute for Mathematical Research, Bristol, UK}
\email{jonathan.chapman@bristol.ac.uk}

\author[M. Ortega]{Miquel Ortega}
\address{Departament de Matemàtiques, Universitat Politècnica de Catalunya, Barcelona, Spain}
\email{miquel.ortega.sanchez-colomer@upc.edu}

\author[V. Souza]{Victor Souza}
\address{Department of Computer Science and Technology, University of Cambridge, UK}
\email{vss28@cam.ac.uk}

\begin{abstract}
The following question was asked by Prendiville:
given an $r$-colouring of the interval $\set{2, \dotsc, N}$, what is the minimum number of monochromatic solutions of the equation $xy = z$?
For $r=2$, we show that there are always asymptotically at least $(1/2\sqrt{2}) N^{1/2} \log N$ monochromatic solutions, and that the leading constant is sharp.
For $r=3$ and $r=4$ we obtain tight results up to a multiplicative logarithmic factor.
We also provide bounds for more colours and other multiplicative equations.
\end{abstract}

\maketitle


\section{Introduction}

Counting monochromatic solutions of equations in coloured sets of integers is one of the prototypical subjects of arithmetic Ramsey theory.
The foundational result of Schur~\cite{Schur1916-lc} states that any $r$-colouring of $[N] \defined \set{1, \dotsc, N}$ has at least one monochromatic solution to the equation $x + y = z$, provided $N$ is large enough as a function of $r$.
The smallest such $N$ is called the \indef{Schur number} $S(r)$.
In fact, an $r$-colouring of $[N]$ is guaranteed to have at least $(c_r - o(1))N^2$ monochromatic solutions to $x + y = z$ for sufficiently large $N$.
This is known as a \indef{supersaturation} result.

Settling a question posed by Graham~\cite{Graham1996-fw}, both Robertson and Zeilberger~\cite{Robertson1998-sy}, and independently Schoen~\cite{Schoen1999-bz} showed that $c_2 = 1/11$ is the optimal constant\footnote{In their work, the constant $1/22$ appears instead of $1/11$ due to the fact that they regard the solutions $x + y = z$ and $y + x = z$ with $x \neq y$ as the same.} for the supersaturated Schur theorem in two colours.
The value of $1/11$ is achieved by colouring the integers $\set{\floor{ 4N/11}, \dotsc, \floor{10N/11}}$ red and the remaining elements of $\set{1, \dotsc, N}$ blue.
Schoen further determined that, up to some minor modifications, this is the only extremal colouring.
Aside from $c_1=1/2$, no other values of $c_r$ are known.

While there has been much work on additive equations, far less is known about multiplicative equations.
Consider the multiplicative analogue of Schur's equation: $x y = z$.
To avoid trivialities, we restrict our attention to solutions with $x, y, z > 1$.
It is a standard observation that if we restrict, say, to the powers of two, then solutions $2^x \cdot 2^y = 2^z$ correspond to solutions of the additive equation $x + y = z$.
Consequently, Schur's theorem implies that if $N$ is sufficiently large relative to $r$, then any $r$-colouring of $[N]$ produces a monochromatic solution to $x y = z$ with $x, y, z > 1$.

Motivated by supersaturation results for additive equations, Prendiville posed the following problem during the 2022 British Combinatorial Conference: what is the minimum number of monochromatic solutions of $x y = z$ in a $3$-colouring of $[2,N]$?
He observed that in the colouring $[2, N] = R \sqcup B \sqcup G$, where $R = [2, N^{1/4}]$, $B = (N^{1/4}, N^{1/2}]$ and $G = (N^{1/2}, N]$, there are only $O(N^{1/4} \log N)$ monochromatic integer solutions of $x y = z$, all in $R$.
In particular, as opposed to the additive case, one can find colourings of $[N]$ where almost all solutions to $x y = z$ are not monochromatic.
This is possible because in a typical solution to the equation $x y = z$ in $[N]$, the numbers $x$ and $y$ are substantially smaller than $z$.
A similar phenomenon was previously observed by Prendiville~\cite{Prendiville2021-oe} for the equation $x - y = z^2$.
On the other hand, by considering powers of two and applying the supersaturated version of Schur's theorem, we can only guarantee $c_r (\log N)^2$ monochromatic solutions in any $r$-colouring of $[N]$.

Prendiville asked this question for three colours because he was interested in techniques that would generalise to multiple colours.
Nonetheless, the same question was also unsolved for two colours.
He provided the colouring $[2,N] = R \sqcup B$ where $R = [2,N^{1/2}]$ and $B = (N^{1/2}, N]$, which has $O(N^{1/2} \log N)$ monochromatic solutions.

For an arbitrary number of colours $r$, one can extend the `restricting to powers of two' argument elucidated above to obtain $c N^{1/S(r)}$ monochromatic solutions, for some $c=c(r)>0$.
Indeed, for any $a > 0$, the set $T_a \defined \set{a, a^2, \dotsc, a^{S(r)}}$ is guaranteed to have a monochromatic solution to $x y = z$ in any $r$-colouring.
For each $2 \leq a \leq N^{1/S(r)}$ which is not a perfect power, we obtain a unique monochromatic solution inside $T_a \subseteq [2,N]$.
As the perfect powers are rather rare, this argument furnishes $c N^{1/S(r)}$ monochromatic solutions to $xy=z$.
For two colours, this provides $cN^{1/5}$ monochromatic solutions.
We remark that Mattos, Mergoni Checchelli and Parczyk~\cite{Mattos2023}, independently have shown, among other results, that any $2$-colouring of $[2,N]$ has at least $N^{1/3}$ monochromatic solutions to $x y = z$.
Our first main theorem improves the exponent in the supersaturated multiplicative Schur theorem from $1/S(r)$ to $1/S(r-1)$.

\begin{theorem}
\label{thm:lower-bound}
For each $r \geq 2$, any $r$-colouring of $\set{2, \dotsc, N}$ contains at least
\begin{equation*}
    c N^{1/S(r-1)}
\end{equation*}
monochromatic solutions to $x y = z$ for large enough $N$, where $c = c(r) > 0$.
Moreover, the exponent $1/S(r-1)$ is sharp for $2 \leq r \leq 4$.
\end{theorem}

An immediate corollary of this theorem is that every $2$-colouring of $[2, N]$ has at least $cN^{1/2}$ monochromatic integer solutions to $x y = z$. The $2$-colouring provided by Prendiville therefore shows that this bound is sharp up to a logarithmic factor.
In view of the results of Robertson--Zeilberger and Schoen for additive equations, it is natural to ask which of the two bounds $N^{1/2}$ and $N^{1/2}\log N$ is closest to the truth, and whether the leading constant $c$ can be determined exactly.
Our next main result answers both of these questions.

\begin{theorem}
\label{thm:two-colours}
Any $2$-colouring of $\set{2, \dotsc, N}$ has a colour class with at least
\begin{equation}
\label{eq:main-lb}
     \p[\bigg]{\frac{1}{2 \sqrt{2}} - o(1)} N^{1/2} \log N
\end{equation}
non-degenerate solutions to $x y = z$ as $N \to \infty$.
\end{theorem}

Therefore, we have established the multiplicative analogue of Graham's problem.
A distinctive feature of this result is that we guarantee the right number of solutions, namely \eqref{eq:main-lb}, in a single colour class.
This is in fact attained by the colouring $[2,N] = R \sqcup B$ with $R = [2, (N/2)^{1/2}] \sqcup (N/2,N]$ and $B = ((N/2)^{1/2}, N/2]$.
Indeed, this colouring only has monochromatic solutions in $R$ and moreover it shows that the constant $1/2\sqrt{2}$ in this theorem is sharp.

An interesting aspect of our proof of \Cref{thm:two-colours} is the use sum-product estimates.
These are deep results in arithmetic combinatorics which, informally, state that if a set of positive integers $A$ has a small sumset $A + A \defined \set{ a + a' \st a, a' \in A}$, then the product set $A \cdot A \defined \set{a a' \st a, a' \in A}$ has to be reasonably large.

Along the way, we establish a stability result (or `inverse theorem') for $2$-colourings of $[2,N]$ with few monochromatic integer solutions to $x y = z$.

\begin{theorem}
\label{thm:stability}
If a $2$-colouring of $\set{2, \dotsc, N}$ has $M$ monochromatic solutions to $x y = z$, then the interval $\set{2, \dotsc, \floor{N/16M}}$ is monochromatic, provided $N$ is sufficiently large.
\end{theorem}

As we obtain a sharp result for up to four colours, \Cref{thm:lower-bound} suggests that the correct exponent of $N$ is related to Schur numbers.
To date, the only known values are $S(1) = 2$, $S(2) = 5$, $S(3) = 14$, $S(4) = 45$, and $S(5) = 161$, the last being found by Huele~\cite{Heule2018-oe} after an extensive computer search.
Our next result provides a connection with a `continuous' variant of the Schur problem.
For each $r \geq 1$, let $I(r)$ be the infimum over all real $T$ such that every $r$-colouring of the closed real interval $[1,T] \subseteq \RR$ has a monochromatic solutions to $x + y = z$.
Note that $I(r) \leq S(r)$ for all $r$.

\begin{theorem}
\label{thm:upper-bound}
For each $r \geq 2$, there exists an $r$-colouring of $\set{2, \dotsc, N}$ with
\begin{equation*}
    C_r N^{1/I(r-1)} \log N + O_r\p[\big]{N^{1/I(r-1)}}
\end{equation*}
monochromatic solutions to $x y = z$ as $N$ grows, where $C_r > 0$ is an explicit constant.
\end{theorem}

The exponent $1/I(r-1)$ in this theorem would be sharp if $I(r) = S(r)$ for all $r$.
While this is true for $1 \leq r \leq 3$, the general case remains open.
See \S\ref{sec:questions} for more details.


\subsection{General multiplicative equations}

We say that an equation is \indef{partition regular} over a set $X$ if any finite colouring of $X$ contains a monochromatic solution of said equation. In the case where $X$ is the set of all positive integers, it is common convention to simply refer to an equation as \indef{partition regular}.
In his celebrated work, Rado~\cite{Rado1933-fe} fully characterised partition regularity for all systems of linear equations.
For example, Rado's criterion \cite[Satz IV]{Rado1933-fe} implies that, for integers $a_i\geq 1$, the equation
\begin{equation}
\label{eq:additive}
    a_1 x_1 + \dotsb + a_k x_k = y
\end{equation}
is partition regular if and only if at least one $a_i$ is equal to $1$.

For a given equation and a fixed number of colours $r$, one can also consider how many monochromatic solutions it is possible to find with respect to an arbitrary $r$-colouring. Given an integer polynomial $P$ in $k$ variables, let $\cN_r(P(x_1,\dotsc,x_k)=0, A)$ denote the minimum number of monochromatic integer solutions to the equation $P(x_1,\dotsc,x_k)=0$ over all colourings of a set $A$ with $r$ colours.
Frankl, Graham and Rödl~\cite{Frankl1988-zl} showed that if \eqref{eq:additive} is partition regular, then
\begin{equation*}
    \frac{\cN_r\p[\big]{a_1 x_1 + \dotsb + a_k x_k = 0, [N]}}{\cN_1\p[\big]{a_1 x_1 + \dotsb + a_k x_k = 0, [N]}} \geq \eps > 0
\end{equation*}
for large enough $N$.
In other words, a positive proportion of all available solutions to \eqref{eq:additive} are guaranteed to be monochromatic in any $r$-colouring of $[N]$.
The aforementioned results of Schoen and Robertson--Zeilberger demonstrate that as $N \to \infty$, we have
\begin{equation*}
    \frac{\cN_2\p[\big]{x + y = z, [N]}}{\cN_1\p[\big]{x + y = z, [N]}}
    = \frac{N^2/11 + O(N)}{N^2/2 + O(N)} = \frac{2}{11} + o(1).
\end{equation*}

Our methods allow us to estimate the number of monochromatic solutions to the multiplicative analogues of \eqref{eq:additive}, namely the equations
\begin{equation}
\label{eq:multiplicative}
    x_1^{a_1} \dotsm x_k^{a_k} = y,
\end{equation}
where $a_1,\dotsc,a_k\geq 1$. It follows from Rado's criterion~\cite[Satz IV]{Rado1933-fe} that this equation is partition regular over $\ZZ_{>1}$ if and only if at least one of the $a_i$ equals $1$.
As in our results concerning the Schur equation, the number of monochromatic solutions to \eqref{eq:multiplicative} is related to the corresponding analogue of the Schur numbers for equation \eqref{eq:additive}.
For $\bfa = (a_1, \dotsc, a_k)$, the \indef{Rado number} $R_\bfa(r)$ for the equation \eqref{eq:additive} is defined to be the minimum $N$ such that any $r$-colouring of $[N]$ has a monochromatic solution to \eqref{eq:additive}.

We call a solution $(x_1,\dotsc,x_k,y)$ to (\ref{eq:multiplicative}) \emph{degenerate} if $x_i=x_j$ holds for some $i\neq j$. When counting solutions to \eqref{eq:multiplicative}, one strategy could be to count the degenerate solutions obtained by setting $x_2 = \dotsb = x_k$.
In other words, solutions to the equation
\begin{equation}
\label{eq:degenerate}
    x_1 x_2^{b} = y,
\end{equation}
where $b = a_2 + \dotsb + a_k$.
In the additive setting, this strategy is necessarily wasteful.
Indeed, the number of solutions to \eqref{eq:additive} with $x_1, \dotsc, x_k, y \in [N]$ is of order $N^{s - 1}$, but setting $x_i = x_j$ for some $i < j$ results in an equation with at most $N^{s - 2}$ solutions.
However, it is possible for a positive proportion of all solutions to be degenerate in the multiplicative setting.
As we show in \Cref{prop:solution-count}, the number of solutions $x_1, \dotsc, x_k, y \in [N]$ to (\ref{eq:multiplicative}) is of order $N (\log N)^{m-1}$, where $m$ is the number of $i$ such that $a_i = 1$.
In particular, if $a_1=1$ and $a_2, \dotsc,  a_k > 1$, then a positive proportion of all solutions to $x_1^{a_1} \dotsm x_k^{a_k} = y$ come from \eqref{eq:degenerate}.

Our next main theorem is the following extension of \Cref{thm:lower-bound} which provides a lower bound for the number of monochromatic non-degenerate solutions to \eqref{eq:multiplicative}.

\begin{theorem}
\label{thm:lower-bound-a}
For any $k \geq 2$, let $\bfa = (a_1, \dotsc, a_k) \in \ZZ_{>0}^k$ and let $m$ be the number of $a_i$ that are equal to $1$.
If $m \geq 1$, then for each $r \geq 2$, any $r$-colouring of $\set{2, \dotsc, N}$ contains at least
\begin{equation*}
    c N^{1/R_\bfa(r-1)}
\end{equation*}
monochromatic non-degenerate solutions to $x_1^{a_1} \dotsm x_k^{a_k} = y$ for large enough $N$, where $c = c_\bfa(r) > 0$.
Moreover, the exponent $1/R_\bfa(r-1)$ is sharp for $2 \leq r \leq 3$.
\end{theorem}

We also obtain an extension of \Cref{thm:upper-bound}.
Generalising the continuous Schur number $I(r)$, we define $I_\bfa(r)$ to be the infimum over all real $T>1$ such that every $r$-colouring of the closed real interval $[1,T] \subseteq \RR$ has a monochromatic solution to \eqref{eq:additive}.

\begin{theorem}
\label{thm:upper-bound-a}
For any $k \geq 2$, let $\bfa = (a_1, \dotsc, a_k) \in \ZZ_{>0}^k$ and let $m$ be the number of $a_i$ that are equal to $1$.
If $m \geq 1$, then for each $r \geq 2$, there exists an $r$-colouring of $\set{2, \dotsc, N}$ with
\begin{equation*}
    \p[\big]{C_{r, \bfa} + o(1)} N^{1/I_{\bfa}(r-1)} (\log N)^{m-1}
\end{equation*}
monochromatic solutions to $x_1^{a_1} \dotsm x_k^{a_k} = y$ as $N$ grows, for an explicit $C_{r, \bfa} > 0$.
\end{theorem}

\Cref{thm:lower-bound-a} and \Cref{thm:upper-bound-a} together imply that as $N \to \infty$, we have
\begin{equation*}
    \p[\big]{c - o(1)} N^{1/R_\bfa(r-1)}
    \leq \cN_r\p[big]{x_1^{a_1} \dotsm x_k^{a_k} = y, [2,N]}
    \leq \p[\big]{C + o(1)} N^{1/I_\bfa(r-1)} (\log N)^{m-1}.
\end{equation*}
Again, if $I_\bfa(r) = R_\bfa(r)$ then these bounds are tight up to a polylogarithmic factor.
In \Cref{prop:rado-two-colour}, we show that this is the case for $r = 1$ and $r = 2$.


\subsection*{Notation}

Let $\NN$ denote the set of positive integers. For $a, b \in \NN$, we use the standard notation $a\mid b$ (respectively, $a\nmid b$) to denote the fact that $a$ divides $b$ (respectively, $a$ does not divide $b$). For positively valued functions $f$ and $g$, we write $f = O(g)$ if there exists a positive constant $C$ such that $f(x) \leq C g(x)$ for all $x$.
We use subscripts, such as in $f=O_\lambda(g)$, to indicate that the implicit constant $C$ may depend on the parameter $\lambda$.
We call a solution $(z_1, \dotsc, z_k)$ to an equation \indef{degenerate} if $z_i = z_j$ holds for some $i \neq j$, and \indef{non-degenerate} otherwise.
We make frequent use of the estimate $\sum_{n=1}^N 1/n = \log N + O(1)$ as $N \to \infty$.


\subsection*{Structure}

The rest of the paper is organised follows.
The first three sections concern monochromatic solutions to $xy=z$.
We begin in \S\ref{secMonoProd} by establishing the general lower bound \Cref{thm:lower-bound}.
In \S\ref{sec:two-colour}, we focus on $2$-colourings and prove the sharp bound \Cref{thm:two-colours} and the stability result \Cref{thm:stability}.
In \S\ref{sec:constructions}, we construct colourings with few monochromatic solutions to $xy=z$ and prove the general upper bound \Cref{thm:upper-bound}.
In \S\ref{sec:multiplicative} we investigate more general multiplicative equations (\ref{eq:multiplicative}).
The main theorems of this section are the lower and upper bounds \Cref{thm:lower-bound-a} and \Cref{thm:upper-bound-a} respectively.
We conclude in \S\ref{sec:questions} by posing some open problems and conjectures.


\section{Monochromatic products}
\label{secMonoProd}

In this section, we consider monochromatic solutions to the equation $x y = z$ with respect to an arbitrary finite number of colours.
As we saw in the introduction, one could show that there are at least $N^{1/S(r)}$ monochromatic solutions by embedding sets of the form $\set{a, a^2, \dotsc, a^{S(r)}}$ into $[2,N]$.
The primary purpose of this section is to prove \Cref{thm:lower-bound}, which amounts to say that we can find roughly $N^{1/S(r-1)}$ monochromatic solutions.
Our gain is obtained by considering different patterns that can be embedded more efficiently.

It turns out that we can get away with an enormous pattern, as long as we can fit many of them inside $[2,N]$.
We use the following two-dimensional set
\begin{equation*}
    M_{a} \defined \set[\big]{ 2^{ji} a^i \st 1 \leq i \leq S,\; 1 \leq j \leq W } \cup \set[\big]{ 2^j \st 1 \leq j \leq W },
\end{equation*}
for some judicious choice of $S$ and $W$.
Writing $M_{a,j}$ for the geometric progression
\begin{equation*}
    M_{a,j} \defined \set[\big]{ (2^j a)^i \st 1 \leq i \leq S },
\end{equation*}
and $M' = \set{2^j \st 1 \leq j \leq W}$ we have $M_a = \bigcup_{1 \leq j \leq W} M_{a,j} \cup M'$.
Moreover, if $a > 1$ is odd, then this is a disjoint union.

The strategy to prove \Cref{thm:lower-bound} is as follows.
We first show that if $S$ and $W$ are large enough, then, for every $a$, any $r$-colouring of the set $M_a$ has a monochromatic non-degenerate solution to $x y = z$ with at most one element in $M'$.
Secondly, we observe that, by varying the value of $a$ in an appropriate set, we are able to embed multiple copies of $M_a$ in $[2,N]$ each of which gives rise to a distinct monochromatic solution of $x y = z$.

A key aspect of the choice of the set $M_a$ is that, putting aside the requirement for the solutions to be distinct, the number of choices of $a$ with $M_a \subseteq[2,N]$ is of order $N^{1/S}/2^W$, as we need $\max M_a = 2^{WS}a^S \leq N$.
Therefore, if we can choose $W$ and $S$ solely as functions of $r$, then no matter how large $W$ is, we have a lower bound of the shape $C_r N^{1/S}$.
To obtain the bound required by \Cref{thm:lower-bound}, we need to choose $S = S(r-1)$, the Schur number for $r-1$ colours.
The choice of $W$ is related to the van der Waerden numbers $W(k,r)$, the smallest value of $N$ such that any $r$-colouring of $[N]$ has a monochromatic arithmetic progression of length $k$.
The van der Waerden theorem~\cite{Van_der_Waerden1927-iu} guarantees that $W(k,r)$ is finite for all $k$ and $r$.

\begin{lemma}
\label{lem:ramsey-pattern-r}
For any $r \geq 2$, set $S = S(r-1)$ and $W = W\p[\big]{S!+1,r^S}$.
Then any $r$-colouring of the set $M_a$ contains a monochromatic non-degenerate solution to $x y = z$ where $z$ is a multiple of $a$.
\end{lemma}
\begin{proof}
Given an $r$-colouring $c \from M_a \to [r]$, consider the auxiliary colouring $\xi \from [W] \to [r]^S$ obtained by colouring an integer $j$ by the colour pattern observed in $M_{a,j}$.
In other words, we define $\xi$ by
\begin{equation*}
    \xi(j) \defined \p[\big]{c(2^j a), c(2^{2j} a^2), \dotsc, c(2^{Sj}a^S)} \in [r]^S.
\end{equation*}
Now, since $W = W\p[\big]{S!+1,r^S}$, there is an arithmetic progression
\begin{equation*}
    \set[\big]{j_0, j_0 + d, j_0 + 2d, \dotsc, j_0 + S!d} \subseteq [W]
\end{equation*}
that is monochromatic in the colouring $\xi$.
That is to say, the sets $M_{a,j_0}, \dotsc, M_{a, j_0 + S!d}$ have the same colour patterns in the colouring $c$, so for every $1 \leq i \leq S$, we have
\begin{equation*}
    c((2^{j_0}a)^i) = c((2^{j_0 + d}a)^i) = \dotsb = c((2^{j_0 + S!d}a)^i).
\end{equation*}
We distinguish two cases, depending on how many colours appears in $M_{a,j_0}$.

\begin{description}[labelindent=\parindent, leftmargin=2\parindent]
\item[First case] no more than $r-1$ colours are used in $M_{a,j_0}$.
Then as $M_{a,j_0}$ is a geometric progression of length $S(r-1)$, there is a monochromatic solution to $x y = z$ in $M_{a,j_0}$.
This solution may be degenerate if we have $x = y = (2^{j_0}a)^{i}$ and $z = xy = (2^{j_0}a)^{2i}$.
In this case, we can find a non-degenerate one by taking
\begin{equation*}
    x' = (2^{j_0}a)^{i}; \qquad
    y' = (2^{j_0 + 2d}a)^{i}; \qquad
    z' = (2^{j_0 + d}a)^{2i}.
\end{equation*}
In any case, we have found a solution where $a \divides z$ and we have not used $M'$ at all.

\item[Second case] all $r$ colours are available in $M_{a,j_0}$.
Then for any element in $M'$, there must be an element in $M_{a,j_0}$ with the same colour.
In particular, we must have $c((2^{j_0} a)^i) = c(2^{S!d})$ for some $1 \leq i \leq S$.
We get a monochromatic non-degenerate solution to $x y = z$ where $a$ divides $z$ by setting
\begin{equation*}
    x = (2^{j_0} a)^{i}, \qquad
    y = 2^{S!d}, \qquad
    z = (2^{j_0 + (S!/i)d} a)^{i}.
\end{equation*}
\end{description}

In both cases, we are able to find the required solution.
\end{proof}

\begin{remark}
    We could have chosen $W = W\p[\big]{\lcm(1,2,\dotsc,S) + 1, r^S}$ instead of $W = W\p[\big]{S! + 1, r^S}$, which is an improvement since $\lcm(1, 2, \dotsc, n) = e^{n + o(n)}$ is much smaller than $n!$. This does not make a significant difference, however, given  the astronomical size of $W$ from the van der Waerden function alone. For more details on the growth of the van der Waerden function, see \cite[\S2.1]{Graham1996-fw}.
\end{remark}

We are now ready to prove \Cref{thm:lower-bound}. The proof that the exponent of $1/S(r-1)$ is sharp is postponed until \S\ref{sec:constructions} and is established by \Cref{lem:small-schur}.

\begin{proof}[Proof of Theorem \ref{thm:lower-bound}]

Choose $S = S(r-1)$ and $W = W(S! + 1, r^S)$.
We are going to find different copies of the set $M_a$ inside $[2,N]$, each of which will lead to a distinct monochromatic non-degenerate copy of $x y = z$.
As previously observed, to guarantee that $M_a \subseteq [2,N]$, we must choose $2 \leq a \leq N^{1/S}2^{-W}$.
To further enforce that we obtain distinct solutions from different copies of $M_a$, we restrict $a$ to the set
\begin{equation*}
    \cA \defined \set[\Big]{ 2 \leq a \leq \floor{N^{1/S}2^{-W}} \st 2 \ndivides a, \text{ $a$ is not a perfect power } }.
\end{equation*}
Each solution obtained by an application of \Cref{lem:ramsey-pattern-r} is guaranteed to have $a \divides z$, that is, $z = 2^{ij} a^i$ for some $1 \leq i \leq S$ and $1 \leq j \leq W$.
Given this value of $z$, we can uniquely recover the value of $a$.
Indeed, as $a$ is odd, we can repeatedly divide $z$ by $2$ to obtain a number of the form $z' = a^i$.
Since no $a \in \cA$ is a perfect power, we can retrieve $a$ from $a^i$ by looking at the smallest integer in the sequence $z', (z')^{1/2}, (z')^{1/3}, \dotsc$.

We have therefore found at least $\card{\cA}$ monochromatic non-degenerate solutions to $x y = z$. Only half of the values of $a \in [2, N^{1/S}2^{-W}]$ are even and a negligible proportion of them are perfect powers. Hence, $\card{\cA} \geq \p[\big]{2^{-W-1} - o(1)}N^{1/S}$ as claimed.
\end{proof}


\section{A sharp result for two colours}
\label{sec:two-colour}

As stated in the introduction, the colouring $[2,N] = R \sqcup B$ with $R = [2, (N/2)^{1/2}] \sqcup (N/2, N]$ and $B = [(N/2)^{1/2}, N/2]$ establishes that
\begin{equation*}
    \cN_2\p[\big]{x y = z , [2,N]} \leq  \frac{1}{2\sqrt{2}}N^{1/2}\log N + O(N^{1/2}).
\end{equation*}
The goal of this section is to prove \Cref{thm:two-colours}, which provides the corresponding lower bound, leading to
\begin{equation*}
    \cN_2\p[\big]{x y = z , [2,N]} = \p[\bigg]{\frac{1}{2 \sqrt{2}} - o(1)}N^{1/2} \log N.
\end{equation*}
In particular, this shows that the aforementioned $2$-colouring is essentially the one with the fewest monochromatic solutions to $xy=z$.

Given integers $a$, $\ell$ and $k$ we define the set
\begin{equation*}
    T_{a, \ell, k} \defined \set{\ell, k, \ell k, a, \ell a,  k a, \ell k a, \ell^2 a, \ell^2 k a}.
\end{equation*}
We begin with the following observation.

\begin{lemma}
\label{lem:ramsey-pattern-two}
For integers $a, \ell, k \geq 2$, consider a $2$-colouring $c \from T_{a, \ell, k} \to \set{\rR, \rB}$.
If $c(\ell) \neq c(k)$, then there exist $x, y, z \in T_{a, \ell, k}$ such that $c(x) = c(y) = c(z)$, $x y = z$ and $a$ divides $z$.
Furthermore, if $a \notin \set{l, k, lk}$, $k \notin \set{ la, l^2a}$, and $l \neq ka$, then we can choose $x,y,z$ to all be distinct.
\end{lemma}
\begin{proof}
Without loss of generality, we may assume that $c(\ell) = \rR$ and $c(k) = \rB$.
Suppose for a contradiction that there is no monochromatic solution to $x y = z$ in $T_{a, k, \ell}$ with $a\mid z$.
We split the proof into two cases according to the colour of $a$.

Firstly, suppose that $c(a) = \rR = c(\ell)$.
Our assumption that there are no monochromatic products leads to the sequence of implications
\begin{equation*}
    c(a)= \rR
    \implies c(\ell a) = \rB
    \implies c(\ell k a) = \rR
    \implies c(\ell^2 k a) = \rB.
\end{equation*}
We therefore obtain a monochromatic solution of the kind $\ell a \cdot \ell k = \ell^2 k a$ or $a \cdot \ell k = \ell k a$ depending on the colour of $\ell k$, thereby delivering a contradiction. If $a\notin\{l,lk\}$ and $k\neq la$ and we instead began by assuming that there are no non-degenerate  monochromatic $xy=z$ with $a\mid z$, then this argument also leads to a contradiction.

Now consider the case where $c(a) = \rB = c(k)$.
Similarly, we have
\begin{align*}
    c(a) = \rB
    &\implies c(k a) = \rR
    \implies c(\ell k a) = \rB
    \implies c(\ell a) = \rR \\
    &\implies c(\ell^2 a) = \rB
    \implies c(\ell^2 k a) = \rR.
\end{align*}
We therefore obtain a monochromatic solution $\ell a \cdot \ell k = \ell^2 k a$ or $a \cdot \ell k = \ell k a$ depending on the colour of $\ell k$, yielding a contradiction. Similarly to the previous case, if $a\notin\{k,lk\}$, $l\neq ka$, and $k\notin\{la,l^2a\}$, then the assumption that there are no non-degenerate monochromatic solutions also leads to a contradiction.
This completes the proof.
\end{proof}

Choosing various values of $a$ and embedding several copies of $T_{a,2,k}$ into $[2,N]$, we can deduce a stability result.
Indeed, we prove \Cref{thm:stability}, which states that if a two-colouring of $[2,N]$ has few monochromatic solutions to $x y = z$, then there must be a long monochromatic initial segment.

\begin{proof}[Proof of \Cref{thm:stability}]
Consider a $2$-colouring $c \from [2, N] \to \set{\rR, \rB}$.
Let $2 < k \leq N$ be minimal such that $c(2) \neq c(k)$.
If no such $k$ exists, then $c$ is monochromatic and we are done.
It suffices to find more than $N / 16k$ monochromatic non-degenerate solutions to $x y = z$.
Indeed, we would then have $N/16k < M$, which gives $k - 1 \geq \floor{N/16M}$, implying the required conclusion that the interval $\set{2, \dotsc, \floor{N/16M}}$ is monochromatic.

The monochromatic interval $[2,k-1]$ furnishes $k\log k + O(k)$ monochromatic solutions to $xy=z$. Since we are taking $N$ sufficiently large, we may henceforth assume $k < N^{1/2}/\loglog N$.

By \Cref{lem:ramsey-pattern-two}, for any integer $a \in [2, N]$ satisfying $4ak \leq N$, there is a monochromatic solution of $x y = z$ in the set
\begin{equation*}
    T_a \defined T_{a, 2, k} = \set{2, k, 2 k, a, 2 a,  k a, 2 k a, 4 a, 4 k a} \subseteq [2,N].
\end{equation*}
Assume initially that $k \neq 4$.
To ensure that the solutions each $T_a$ provide are different and non-degenerate, we restrict ourselves to values of $a$ in the set
\begin{equation*}
    \cA \defined \set[\Big]{a \in [2,N/4k] \st 2 \ndivides a, \; k \ndivides 4a }.
\end{equation*}
For each $a \in \cA$, we use \Cref{lem:ramsey-pattern-two} to find a monochromatic non-degenerate solution to $x y = z$ with $a \divides z$.
We claim that all these solutions are different.
Indeed, given $z$ in one such solution, we can recover the unique value of $a \in \cA$ with $z \in T_a$.
We do so by noting that either $z \in \set{ka, 2ka, 4ka}$ or $z \in \set{a, 2a, 4a}$.
In the former case, we divide $z$ by $k$ exactly once. Since $k$ does not divide $4a$, $2a$ or $a$, this leads to the conclusion that $z/k \in \set{a, 2a, 4a}$.
Having obtained $z$ or $z/k$ in $\set{a, 2a, 4a}$, since $a$ is odd, we can now retrieve $a$ by repeatedly dividing by $2$.

This gives us at least $\card{\cA}$ monochromatic non-degenerate solutions to $x y = z$.
It remains to show that $\card{\cA} > N/ 16k$.
Let $h=k/\gcd(4,k)$.
Note that if $k \divides 4a$, then $h \divides a$, therefore
\begin{equation*}
    \card{\cA}
    \geq \card[\big]{\set[\big]{a \in [2, N/4k] \st 2 \ndivides a, \; h \ndivides a }}.
\end{equation*}
Since $k \notin \set{1,2,4}$, we have $h \geq 2$. If $h$ is a power of $2$, then $\cA$ consists of all the odd numbers in $[2, N/4k]$.
Since $k < N^{1/2}/\loglog N$ and $N$ is sufficiently large, we have
\begin{equation*}
    \card{\cA} \geq \frac{N}{8k} - 1 > \frac{N}{16k}.
\end{equation*}
If $h$ is not a power of $2$, then $h \geq 3$ and $\cA$ is precisely all the odd numbers in $[2, N/4k]$ which are not congruent to $h$ modulo $2h$.
Again using the fact that $h \leq k < N^{1/2}/\loglog N$, for large enough $N$ this implies that
\begin{equation*}
    \card{\cA} \geq \frac{N}{8k} \cdot \left(\frac{h-1}{h}\right) -O\p[\bigg]{\frac{N^{1/2}}{\loglog N}} > \frac{N}{16 k}.
\end{equation*}

If $k = 4$, the sets $T_a$ are of the form $\set{2, 4, 8, a, 2 a, 4 a, 8 a, 16 a}$, which means that a solution could be counted in $T_a$ and $T_{2a}$.
This time, we take $a$ in the set
\begin{equation*}
    \cA' \defined \set[\Big]{a \in [2, N/16] \st 2 \ndivides a}.
\end{equation*}
By repeated division by $2$, we can recover $a$ from $z$ as before.
By applying \Cref{lem:ramsey-pattern-two}, each $T_a$ with $a \in \cA'$ gives a distinct monochromatic non-degenerate solution to $x y = z$.
But then we have $\card{\cA'} \geq N/16 - 3/2 \geq N/8k$ for large $N$, and we are done.
\end{proof}

The constant $1/16$ in \Cref{thm:stability} is not the best possible we could obtain.
With a careful analysis in the case $h = 3$ above, we could replace $N/16k$ by $N/12k + O(1)$.
Moreover, if we assumed that $k\to\infty$ as $N\to\infty$, we could have replaced $N/16k$ by $(1 + o(1)) N/8k$.
As this value is not important in the proof of \Cref{thm:two-colours}, we did not optimise for it.

While we have this stability result for two colours, we do not obtain a similar statement for three colours.
See \S\ref{sec:questions} for some further questions and comments about stability results in the more general setting.

Before we prove \Cref{thm:two-colours}, we first record two results which are required for the proof.
The first is the classical \indef{divisor bound}, which asserts that $d(n) = n^{o(1)}$, where $d(n)$ is the number of positive integers which divide $n$.
More refined estimates are available, but the following is sufficient for our purposes.

\begin{proposition}[Divisor bound]
\label{prop:divisor-bound}
For every $\eps > 0$ and large enough $n \geq n_0(\eps)$, we  have $d(n) \leq n^\eps$.
\end{proposition}

The second result we require is significantly deeper.
It concerns the sizes of sum sets $A+A \defined \set{a + a' \st a,a' \in A}$ and product sets $A \cdot A \defined \set{ a a' \st a, a' \in A}$ for finite sets of integers $A$.

\begin{theorem}[Sum-product bound]
\label{thm:sumproduct}
There is $\delta > 0$ such that every finite set of real numbers $A$ with $|A|\geq 2$ satisfies
\begin{equation*}
    \max \set[\big]{\card{A+A}, \card{A \cdot A}} \geq \card{A}^{1 + \delta}.
\end{equation*}
\end{theorem}

Erd\H{o}s and Szemerédi~\cite{Erdos1983-lo} were the first to prove such a result for sets of integers, for a very small and inexplicit value of $\delta > 0$.
They conjectured that, provided $\card{A}$ is sufficiently large in terms of $\delta$, this result should hold for any $0 < \delta < 1$.
Building upon the approach of Solymosi~\cite{Solymosi2005-pq}, the current record of $0 < \delta < 1/3 + 2/1167$ is due to Rudnev and Stevens~\cite{Rudnev2022-fj}.
The weaker estimate of $\delta > 0$ will suffice in our application of \Cref{thm:sumproduct}.
We refer the interested reader to the introduction of \cite{Rudnev2022-fj} for details on the history of the Erd\H{o}s-Szemerédi sum-product conjecture.

We are now ready to prove \Cref{thm:two-colours}.

\begin{proof}[Proof of \Cref{thm:two-colours}]
Let $c \from [2,N] \to \set{\rR, \rB}$ be a colouring.
Write $M_\rR$ and $M_\rB$ respectively for the number of non-degenerate monochromatic solutions to $x y = z$ in colours $\rR$ and $\rB$.
The goal is to show that $\max\set{M_\rR, M_\rB}$ is at least the quantity in \eqref{eq:main-lb}.
This is definitely the case if $c$ is a constant colouring.
Assume then that $c(2) = \rR$ and that $k$ is the minimum $2 < k \leq N$ with $c(k) = \rB$.
In particular, the interval $[2,k-1]$ is monochromatic in $\rR$, thus we are done unless $k < (N/2)^{1/2}$.

From \Cref{thm:stability}, we have $M_\rB + M_\rR \geq N/16k$.
Therefore, if  $ k \leq N^{1/2}/ 32 \log N$, then $\max\set{M_\rR, M_\rB} \geq N/32k \geq N^{1/2} \log N$.
In other words, we are done unless
\begin{equation*}
    \frac{N^{1/2}}{32 \log N} < k < (N/2)^{1/2}.
\end{equation*}

Motivated by the fact that in the extremal colouring, all monochromatic solutions lie in the initial interval $[2,(N/2)^{1/2}]$, we will try to find many solutions using pairs $(i,j) \in [2,k-1]^2$ with $ij \leq (N/2)^{1/2}$.
For each such pair $(i,j)$, we find a solution if $c(ij) = \rR$ as we have $c(i) = c(j) = \rR$.
For technical reasons, we restrict the pairs
\begin{equation*}
    \cV \defined \set[\Big]{ (i,j) \in [2,k-1]^2 \st  32 \log N \leq  i < k/2 \;,\; i j \leq (N/2)^{1/2}}.
\end{equation*}
Note that we have plenty of pairs in $\cV$.
Indeed, we have
\begin{align*}
    \card{\cV}
    &\geq \sum_{i = \ceil{32 \log N}}^{\floor{k/2}} \p[\bigg]{\floor[\Big]{\frac{(N/2)^{1/2}}{i}} - 1 }
    = (N/2)^{1/2} \sum_{i = \ceil{32 \log N}}^{\floor{k/2}} \frac{1}{i} + O(k) \\
    &= \frac{1}{2 \sqrt{2}} N^{1/2} \log N + O\p[\big]{N^{1/2} \loglog N},
\end{align*}
as $N \to \infty$.
Our goal now is to show that either $M_\rB \geq N^{1/2 + \eps}$ or $M_\rR \geq \p[\big]{1 - o(1)}\card{\cV}$.

Define the equivalence relation $\sim$ on $\cV$ such that $(i,j) \sim (i', j')$ whenever $j' = j$ and $i' = 2^t i$ for some $t \in \ZZ$.
Write $\cC = \cV / \sim$ for the set of equivalence classes and note that each $C \in \cC$ can be written in the form
\begin{equation}
\label{eq:class}
    C = \set[\big]{(i,j), (2i,j), \dotsc, (2^{\card{C}-1} i, j)}.
\end{equation}
For each equivalence class as above, consider the following set of $\card{C} + 1$ numbers
\begin{equation*}
    ij, \quad 2ij, \quad 2^2 ij, \quad \dotsb \quad 2^{\card{C}-1}ij, \quad 2^{\card{C}}ij,
\end{equation*}
and note that for each of these numbers that are coloured in $\rR$, we have a solution of the form $(x,y,z)$ with $x = 2^{t}i$, $y = j$ and $z = 2^{t}ij$.
Indeed, as $(2^{\card{C}-1}i,j) \in \cV$, we have $2^{\card{C}}i < k$ and $j < k$, so $c(j) = c(2^{t} i) = \rR$ for all $0 \leq t \leq \card{C}$.

An equivalence class $C \in \cC$ as in \eqref{eq:class} is said to be \indef{abundant} if there is at most one $0 \leq t \leq \card{C}$ with $c(2^t i j) =\rB$.
Each such class contributes at least $\card{C}$ solutions in colour $\rR$, which are moreover distinct, since they all have $(x,y,z) = (i,j,ij)$ for some $(i,j) \in [2,k-1]^2$ with $ij \leq (N/2)^{1/2}$.
Thus, writing $\cA$ for the set of abundant equivalence classes, we have
\begin{align}
\label{eq:equivalence-count}
    M_\rR \geq \sum_{C \in \cA} \card{C} = \card{\cV} - \sum_{C \in \cC \setminus \cA} \card{C}.
\end{align}
Therefore, we are done if $\sum_{C \in \cC \setminus \cA} \card{C} = o(\card{\cV})$.
Notice that a class $C \in \cC$ is abundant if it does not intersect any set of pairs of the form
\begin{equation*}
    \cV_h \defined \set[\Big]{ (i,j) \in \cV \st  (2^h i, j) \in \cV \;,\; c(ij) = c(2^{h+1} ij) = \rB },
\end{equation*}
for $h \geq 0$.
Note also that $\card{\cV_h} = \emptyset$ if $h +1\geq \log N$ and that $\card{C} \leq \log N$ for all $C \in \cC$.
Therefore, we have
\begin{equation*}
    \sum_{C \in \cC \setminus \cA} \card{C}
    \leq \card{\cC\setminus \cA} \log N
    \leq \log N \sum_{h \geq 0} \card{\cV_h}
    \leq (\log N)^2 \max_{h \geq 0} \card{\cV_h}.
\end{equation*}
With some foresight, we pick some $0 < \eps < 1/2$ such that $(1 + \delta)(1/2 - \eps) = 1/2 + \gamma$, for some $\gamma > 0$, where $0 < \delta < 1$ is admissible in \Cref{thm:sumproduct}.
For concreteness, we can take $\delta = 1/3$ and $\eps = 1/12$, which would then give $\gamma = 1/18$.
Thus, from \eqref{eq:equivalence-count} we have $M_\rR \geq (1 - o(1))\card{\cV}$ if we have $\card{\cV_h} \leq N^{1/2 - \eps}$ for all $h \geq 0$.

Finally, suppose $\card{\cV_h} \geq N^{1/2 - \eps}$ for some $h \geq 0$. We use this assumption to show that $M_\rR + M_\rB$ must be rather large.
Consider the set $A_h \defined \set[\big]{ ij \st (i,j) \in \cV_h }$.
Each element $a \in A_h$ can be represented in no more than $d(a)$ ways as $a = ij$.
By the divisor bound, we have $d(a) \leq N^{o(1)}$, so, in particular, we have $N^{1/2 - \eps - o(1)}\leq \card{A_h} \leq (N/2)^{1/2}$.
Our choice of parameters therefore implies that $N^{1/2} \leq \card{A_h}^{1+\delta +o(1)}N^{-\gamma} \leq \card{A_h}^{1+\delta-\gamma + o(1)}$.
Since $A_h\subseteq[2,(N/2)^{1/2}]$, we have $\card{A_h + A_h} \leq 2 N^{1/2} \leq \card{A_h}^{1 + \delta - \gamma + o(1)}$.
By the sum-product bound (Theorem \ref{thm:sumproduct}), we infer that $\card{A_h \cdot A_h} \geq \card{A_h}^{1 + \delta} \geq N^{1/2 + \gamma - o(1)}$.

We now demonstrate that each element in $A_h \cdot A_h$ leads to a monochromatic solution.
Let $w \in A_h \cdot A_h$, where $w = i j \cdot i' j'$ for some $(i, j), (i', j') \in \cV_h$.
This implies, for instance, that $c(i) = c(j) = c(i') = c(j') = c(2^h i) = c(2^h i') = \rR$ and that $c(ij) = c(i'j') = c(2^{h+1}ij) = c(2^{h+1}i'j') = \rB$.
Moreover, we have that $(2^h i, j), (i', j') \in \cV$, so $2^h ij ,\; i'j' \leq (N/2)^{1/2}$, and thus $2^{h+1}iji'j' \leq N$.
Finally, note that we also have $c(2^{h+1}) = \rR$ as $2^{h+1} < k$ follows from the fact that $(N/2)^{1/2} \geq 2^h i j \geq 2^{h+1} \cdot 32 \log N$.
Now if $c(i j i'j') = \rB$, then we procure a solution in $\rB$ of the form
\begin{equation*}
    (x,y,z) = (i j, \; i'j',\; i j \cdot i'j').
\end{equation*}
Similarly, if $c(2^{h+1} i j i'j') = \rB$, then we obtain a solution in $\rB$ of the form
\begin{equation*}
    (x,y,z) = (2^{h+1} i j, \; i'j',\; 2^{h+1} i j \cdot i'j').
\end{equation*}
Finally, if $c(ij i'j') = c(2^{h+1} i j i'j') = \rR$, then we find a solution in $\rR$ of the form
\begin{equation*}
    (x,y,z) = (2^{h+1}, \;  i j i'j',\; 2^{h+1} \cdot i j i'j').
\end{equation*}
In any case, we always obtain a solution with $z = w$ or $z = 2^{h+1}w$.
Since a solution can only be counted at most twice in this manner, we deduce that $M_\rR + M_\rB \geq \card{A_h \cdot A_h}/2$.

We therefore conclude that
\begin{equation*}
     \max\set{M_\rR, M_\rB} \geq \card{A_h \cdot A_h}/4 = N^{1/2 + \gamma - o(1)},
\end{equation*}
and we are done.
\end{proof}

It would be interesting to see whether the ideas presented in the proof of \Cref{thm:two-colours} can be extended to more than $2$ colours or applied to other equations.
See \S\ref{sec:questions} for related open problems.


\section{Colourings with few monochromatic products}
\label{sec:constructions}

In this section, we provide constructions of colourings of $\set{2, \dotsc, N}$ with relatively few monochromatic solutions to $xy=z$. The main result of this section is the upper bound \Cref{thm:upper-bound}.

Recall that $\cN_r(xy=z, A)$ denotes the minimum number of monochromatic integer solutions to $xy=z$ over all $r$-colourings of the set $A$.
If we write $d(n)$ the number of divisors of $n$, we have the standard estimate
\begin{equation*}
    \sum_{n=1}^{N} d(n) = N \log N + O(N),
\end{equation*}
which in turn implies that
\begin{equation*}
    \cN_1(xy=z, [2,N])
    = \cN_1(xy=z, [N]) - N
    = N\log N + O(N).
\end{equation*}
The same bound also applies for the number of non-degenerate solutions, since the equation $x \cdot x = z$ has only $\floor{N^{1/2}}$ solutions over $[N]$.

As with the examples of Prendiville mentioned in the introduction, most of our constructions will be interval colourings.
Recall Prendiville's $2$-colouring $[2,N] = R \sqcup B$, where $R = [2, N^{1/2}]$ and $B = (N^{1/2}, N]$. We depict this graphically as:
\begin{equation}
\label{col:prendiville-2}
\begin{tikzpicture}[baseline=(current  bounding  box.center), scale=6.5]
\clip (-0.1,0.04) rectangle (0.4+0.1,-0.11);
\foreach \x/\xtext in
{0/$2$, 1/$N^{\frac{1}{2}}$, 2/$N$}
\draw (\x*0.2,-0.02) node[below] {\strut \xtext};
\ccint{0}{0.2}{Red};
\ocint{0.2}{0.4}{Cerulean};
\end{tikzpicture}
\end{equation}
In this colouring, the only solutions to $x y = z$ are red and lie in the interval $[2, N^{1/2}]$.
By the computation above, there are $(1/2) N^{1/2} \log N + O\p[\big]{N^{1/2}}$ monochromatic solutions in the colouring \eqref{col:prendiville-2}.
We also saw that we can improve the leading constant with a small modification: $R = [2, (N/2)^{1/2}] \sqcup (N/2, N]$ and $B = ((N/2)^{1/2}, N/2]$.
Graphically,
\begin{equation}
\label{col:improved-2}
\begin{tikzpicture}[baseline=(current  bounding  box.center), scale=6.5]
\clip (-0.1,0.04) rectangle (0.6+0.1,-0.13);
\foreach \x/\xtext in
{0/$2$, 1/$(\frac{N}{2})^{\frac{1}{2}}$, 2/$\phantom{((}\frac{N}{2}\phantom{)^{\frac{1}{2}}}$, 3/$N$}
\draw (\x*0.2,-0.02) node[below] {\strut \xtext};
\ccint{0}{0.2}{Red};
\ocint{0.2}{0.4}{Cerulean};
\ocint{0.4}{0.6}{Red};
\end{tikzpicture}
\end{equation}
Again, the only monochromatic solutions to $x y = z$ are red and lie in the now shortened initial interval $[2, (N/2)^{1/2}]$.
This leads to an improved constant of $1/(2\sqrt{2})$, which implies the sharpness of \Cref{thm:two-colours}.

For three colours, Prendiville suggested the very natural generalisation of \eqref{col:prendiville-2}:
\begin{equation}
\label{col:prendiville-3}
\begin{tikzpicture}[baseline=(current  bounding  box.center), scale=6.5]
\clip (-0.1,0.04) rectangle (0.6+0.1,-0.11);
\foreach \x/\xtext in
{0/$2$, 1/$N^{\frac{1}{4}}$, 2/$N^{\frac{1}{2}}$, 3/$N$}
\draw (\x*0.2,-0.02) node[below] {\strut \xtext};
\ccint{0}{0.2}{Red};
\ocint{0.2}{0.4}{Cerulean};
\ocint{0.4}{0.6}{ForestGreen};
\end{tikzpicture}
\end{equation}
Explicitly, we partition $[2,N] = R \sqcup B \sqcup G$ with $R = [2, N^{1/4}]$, $B = (N^{1/4}, N^{1/2}]$ and $G = (N^{1/2}, N]$.
This colouring has $(1/4)N^{1/4} \log N + O\p[\big]{N^{1/4}}$ monochromatic solutions, which, as \Cref{thm:lower-bound} shows, is far from optimal.
Indeed, the following $3$-colouring outperforms it significantly:
\begin{equation}
\label{col:schur-3}
\begin{tikzpicture}[baseline=(current  bounding  box.center), scale=6.5]
\clip (-0.1,0.04) rectangle (1.0+0.1,-0.11);
\foreach \x/\xtext in
{0/$2$, 1/$N^{\frac{1}{5}}$, 2/$N^{\frac{2}{5}}$, 4/$N^{\frac{4}{5}}$, 5/$N$}
\draw (\x*0.2,-0.02) node[below] {\strut \xtext};
\ccint{0}{0.2}{Red};
\ocint{0.2}{0.4}{Cerulean};
\ocint{0.4}{0.8}{ForestGreen};
\ocint{0.8}{1.0}{Cerulean};
\end{tikzpicture}
\end{equation}
In other words, we partition $[2, N] = R \sqcup B \sqcup G$ into $R = [2, N^{1/5}]$, $B = (N^{1/5}, N^{2/5}] \sqcup (N^{4/5}, N]$ and $G = (N^{2/5}, N^{4/5}]$.
The only monochromatic solutions to $x y = z$ in \eqref{col:schur-3} are in the first red interval, and there are $(1/5) N^{1/5} \log N + O\p[\big]{N^{1/5}}$ of them.
Again, we could do some slight modifications in this colouring to reduce the leading constant.
\begin{equation}
\label{col:improved-3}
\begin{tikzpicture}[baseline=(current  bounding  box.center), scale=6.5]
\clip (-0.1,0.04) rectangle (1.6+0.1,-0.15);
\foreach \x/\xtext in
{0/$2$, 1/$\frac{N^{1/5}}{2^{4/5}}$, 2/$\frac{N^{2/5}}{2^{8/5}}$, 3/$\frac{N^{2/5}}{2^{3/5}}$, 5/$\frac{N^{4/5}}{2^{6/5}}$, 6/$\frac{N^{4/5}}{2^{1/5}}$,7/$\frac{N}{2}$,8/$N$}
\draw (\x*0.2,-0.02) node[below] {\strut \xtext};
\ccint{0}{0.2}{Red};
\ocint{0.2}{0.4}{Cerulean};
\ocint{0.4}{0.6}{Red};
\ocint{0.6}{1.0}{ForestGreen};
\ocint{1.0}{1.2}{Red};
\ocint{1.2}{1.4}{Cerulean};
\ocint{1.4}{1.6}{Red};
\end{tikzpicture}
\end{equation}
The colouring indicated in \eqref{col:improved-3} only has monochromatic solutions in the first red interval, but this time, there are $\p[\big]{1/(5 \cdot 2^{1/5})} N^{1/5} \log N + O\p[\big]{N^{1/5}}$ of them.
While the inclusion of these multiple red intervals reduces the constant, the essence of \eqref{col:improved-3} is captured by the colouring \eqref{col:schur-3},  which is constructed from an auxiliary colouring $c \from \set{1, 2, 3, 4} \to \set{\rB, \rG}$ with $c(1) = c(4) = \rB$ and $c(2) = c(3) = \rG$.
The auxiliary colouring has no monochromatic solutions to the Schur equation $x + y = z$, suggesting that using a Schur colouring in the logarithmic scale is at least sometimes sharp.

\subsection{Interval Schur numbers}

Recall the definition of the Schur numbers
\begin{equation*}
    S(r) \defined \min \set[\Bigg]{ N \in \NN \st
    \begin{array}{c}
    \text{for any $r$-colouring of the integer interval $[N]$ there} \\
    \text{are $x, y, z \in [N]$ of the same colour with $x + y = z$}
    \end{array}
    }.
\end{equation*}
The only known values of $S(r)$ are $S(1) = 2$, $S(2) = 5$, $S(3) = 14$, $S(4) = 45$ and $S(5) = 161$.
A related quantity that will be useful in our constructions is the following continuous analogue of the Schur numbers, which we call \indef{interval Schur numbers}
\begin{equation*}
    I(r) \defined \inf \set[\Bigg]{ T \in \RR \st
    \begin{array}{c}
    \text{for any $r$-colouring of the real interval $[1,T]$ there} \\
    \text{are $x, y, z \in [1,T]$ of the same colour with $x + y = z$}
    \end{array}
    }.
\end{equation*}
By definition, for any $T' < I(r)$, there is an $r$-colouring of the real interval $[1,T']$ without any monochromatic solution to $x + y = z$.
Moreover, notice that $I(r) \leq S(r)$ and $I(r) \geq 2$ for all $r \geq 1$.

By applying the compactness principle, we can derive an alternative definition of $I(r)$ as the largest $T \in \RR$ for which there is an $r$-colouring of the half-open real interval $(1,I(r)]$ with no monochromatic solution to $x + y = z$.
This reformulation will be utilised in our proof of \Cref{thm:upper-bound}.
A similar argument can be used to replace $(1,I(r)]$ with $[1,I(r))$ in this alternative definition.

\begin{lemma}
\label{lem:compact}
For all $r \geq 1$, there is an $r$-colouring of the half-open real interval $(1,I(r)]$ with no monochromatic solution to $x + y = z$.
Moreover, for all $T > I(r)$, no such $r$-colouring exists for $(1,T]$.
\end{lemma}
\begin{proof}
First we show that if $T > I(r)$, no $r$-colouring of $(1,T]$ is devoid of monochromatic solutions to $x + y = z$.
Let $I(r) < T' < T$ and set $\alpha = T / T' > 1$.
By the definition of $I(r)$, any $r$-colouring of the real interval $[1,T']$ yields a monochromatic solution to $x + y = z$.
The same is true for $\alpha \cdot [1, T'] = [\alpha, T] \subseteq (1, T]$, as claimed.

It remains to find an $r$-colouring of $(1, I(r)]$ with no monochromatic solution to $x + y = z$. Suppose for a contradiction that no such colouring exists.
By the compactness principle\footnote{Note that this argument uses some form of the axiom of choice.} \cite[\S 1.5, Theorem 4]{Graham2013-rl} there exists a finite non-empty set $F \subseteq (1, I(r)]$ such that every $r$-colouring of $F$ produces a monochromatic solution to $x + y = z$.
Writing $m$ and $M$ for the smallest and largest elements of $F$ respectively, we notice that the set $F/m \defined \set{x/m \st x \in F}$ also has the property that every $r$-colouring produces a monochromatic solution to $x + y = z$.
The same is also true for the real interval $[1, M/m]$ since it contains $F/m$.
As $m > 1$ and $M \leq I(r)$, we have $M/m < I(r)$, which contradicts the definition of $I(r)$.
\end{proof}

We now return to the task of proving \Cref{thm:upper-bound}.
As hinted by the colouring \eqref{col:schur-3}, the key idea behind our approach is to convert a multiplicative Ramsey problem into an additive one.

\begin{proof}[Proof of Theorem \ref{thm:upper-bound}]
Recall that our goal is to show that for every $r \geq 2$ and large enough $N$, there is a colouring of $[2, N]$ with at most $C_r N^{1/I(r-1)} \log N + O\p[\big]{N^{1/I(r-1)}}$ solutions to $x y = z$, for some constant $C_r > 0$.

By \Cref{lem:compact}, there is an $(r-1)$-colouring $\xi$ of the real interval $(1, I(r-1)]$ without monochromatic solutions to $x + y = z$.
We build an $r$-colouring of the integer interval $[2,N]$ in the following way.
For $M = N^{1/I(r-1)}$, we colour the interval $[2,M]$ with the new colour $r$ and use $\xi$ to fill the interval $(M,N]$ as follows:
\begin{equation*}
    c(x) \defined \begin{cases}
        r & \text{if $2 \leq x \leq M$,} \\
        \xi\p[\big]{\log x / \log M} & \text{if $M < x \leq N$.}
    \end{cases}
\end{equation*}
Since every $x \in (M, N]$ satisfies $\log x / \log M \in (1, I(r-1)]$, this defines a colouring of the entire real interval $[2,N]$.

We claim that the colouring $c$ has no monochromatic solutions to $x y = z$ with colours lying in $\set{1, \dotsc, r - 1}$.
Indeed, if we had such a solution, then we must have $M < x, y, z \leq N$ with $x = M^{x'}$, $y = M^{y'}$ and $z = M^{z'}$ for some $x', y', z' \in (1, I(r-1)]$.
But $c(x) = \xi(x') = c(y) = \xi(y') = c(z) = \xi(z')$, and the equality $x y = z$ implies that $x' + y' = z'$, contradicting our choice of $\xi$.

Therefore, all monochromatic solutions occur in colour $r$ and must come from the interval $[2,M]$ alone.
Therefore, $c$ contains $M \log M + O(M)$ monochromatic solutions to $x + y = z$.
However,
\begin{align*}
    M \log M + O(M)
    &= \frac{1}{I(r-1)} N^{1/I(r-1)} \log N + O\p[\big]{N^{1/I(r-1)}},
\end{align*}
so we can choose $C_r = 1/I(r-1)$ and we are done.
\end{proof}

\subsection{A related Schur number}

In \Cref{thm:lower-bound}, we guarantee that any $r$-colouring of $[2, N]$ has at least $c_r N^{1/S(r-1)}$ monochromatic solutions to $x y = z$. Conversely, our constructions for \Cref{thm:upper-bound} provide at most $ C_r N^{1/I(r-1)} \log N$ monochromatic solutions.
It therefore becomes of prime importance to understand whether $I(r) = S(r)$.
Curiously, our investigations into this matter lead us to the following variant of the Schur numbers
\begin{equation*}
    S^*(r) \defined \min \set*{ N \in \NN \st
    \begin{array}{c}
    \text{for any $r$-colouring of the integer interval $[N]$} \\
    \text{there are $x, y, z \in [N]$ of the same colour with} \\
    \text{either $x + y = z$ or $x + y + 1 = z$}
    \end{array}
    }.
\end{equation*}
This parameter was introduced by Abbott and Hanson~\cite[\S4]{Abbott1972-sm} for a different purpose.
The significance of $S^*(r)$ in our work stems from the following observation.

\begin{lemma}
\label{lem:rel-schur}
For every $r \geq 1$, we have $S^*(r) \leq I(r)\leq S(r)$.
\end{lemma}
\begin{proof}
The bound $I(r)\leq S(r)$ follows immediately from the fact that any real interval $[1,T]$ which admits a sum-free $r$-colouring cannot contain the integer interval $[S(r)]$.

We now show that $S^*(r) \leq I(r)$.
Let $S = S^*(r) - 1$.
By definition, there is an $r$-colouring $\xi$ of the integer interval $[S]$ without monochromatic solutions to $x + y = z$ and $x + y + 1 = z$.
We now define a $r$-colouring $c$ of the real interval $[1,S+1)$ by
\begin{equation*}
    c(x) = \xi\p[\big]{\floor{x}}.
\end{equation*}
Now suppose that the numbers $x$, $y$ and $z$ in the real interval $[1,S)$ satisfy $x + y = z$.
Hence,
\begin{equation*}
    \floor{z} = \floor{x} + \floor{y} \,\text{ or }\, \floor{z} = \floor{x} + \floor{y} + 1.
\end{equation*}
In either case, the colours $c(x)$, $c(y)$ and $c(z)$ cannot all be the same.
Therefore, $I(r) \geq S+1 = S^*(r)$ as claimed.
\end{proof}

From the proof, we observe that a colouring of an integer interval that avoids monochromatic solutions to both $x + y = z$ and $x + y + 1 = z$ can be extended to a colouring of a continuous interval by replacing an element $k$ by a monochromatic real interval $[k,k+1)$ of the same colour as $k$.
The condition that there is no solution to $x + y + 1 = z$ handles `rounding issues' coming from the fact that $\lfloor x+y\rfloor=\lfloor x\rfloor + \lfloor y\rfloor + 1$ can occur when $x$ and $y$ are near the right-endpoint of their corresponding interval.
In fact, the main method we have to produce valid colourings to estimate $I(r)$ is to exhibit a colouring for $S^*(r)$ and extending it in this way.

When describing a colouring $c \from [N] \to \set{\rR, \rB}$, we sometimes write a length $N$ word over the alphabet $\set{\rR, \rB}$.
That is, we write $\rR \rR \rB$ for the colouring with $c(1) = \rR$, $c(2) = \rR$ and $c(3) = \rB$.
We use a similar convention to describe colourings with an arbitrary number of colours.

\begin{lemma}
\label{lem:small-schur}
We have $S^*(r) = I(r) = S(r)$ for $r = 1, 2, 3$.
\end{lemma}
\begin{proof}
In view of \Cref{lem:rel-schur}, it suffices to show that $S^*(r) = S(r)$ for all $r \in \set{1, 2, 3}$.
As the values of $S(r)$ are known for $r \in \set{1, 2, 3}$, it is enough to show that some of the $r$-colourings of $[S(r) - 1]$ without monochromatic solutions to $x + y = z$ also avoid monochromatic solutions to $x + y + 1 = z$.
For $r = 1$, we have $S(1) = 2$ and the only colouring of $[1]$ with one colour has no monochromatic solutions to $x + y = z$ or $x + y + 1 = z$, so $S^*(1) = 2$.

For $r = 2$, we have $S(2) = 5$.
Moreover, up to permuting colours, the only colouring $c \from [4] \to \set{\rR, \rB}$ without monochromatic solutions to $x + y = z$ is $\rR \rB \rB \rR$, which also has no monochromatic solution to $x + y + 1 = z$.
Thus, $S^*(2) = 5$.

For $r = 3$, we have $S(3) = 14$.
There are exactly three colourings $c \from [13] \to \set{\rR, \rB, \rG}$, up to permuting colours, without monochromatic solutions to $x + y = z$. They are
\begin{align*}
&\rR\rB\rB\rR\rG\rG\rR\rG\rG\rR\rB\rB\rR, \\
&\rR\rB\rB\rR\rG\rG\rB\rG\rG\rR\rB\rB\rR, \\
&\rR\rB\rB\rR\rG\rG\rG\rG\rG\rR\rB\rB\rR.
\end{align*}
In other words, colourings of the form $\rR\rB\rB\rR\rG\rG\bfx\rG\rG\rR\rB\rB\rR$, where $\bfx$ is any colour.
If $\bfx = \rB$, then $x = 3$, $y = 7$ and $z = 11$ would be a monochromatic solution to $x + y = z$.
However, if $\bfx \in \set{\rR, \rG}$, then there are no monochromatic solutions to either $x + y = z$ nor $x + y + 1 = z$, thus $S^*(3) = 14$.
\end{proof}

We now investigate which $r$-colourings of the integer interval $[2, N]$ can be constructed from \Cref{thm:lower-bound} and \Cref{lem:small-schur}.
If $r = 3$, then $I(2) = S^*(2) = 5$, where the colouring $\rR\rB\rB\rR$ can be turned into a colouring of the real interval $[1,5)$ as
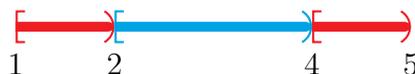
\begin{equation*}
\begin{tikzpicture}[baseline=(current  bounding  box.center), scale=6.5]
\clip (-0.1,0.04) rectangle (0.8+0.1,-0.1);
\foreach \x/\xtext in
{0/$1$, 1/$2$, 3/$4$, 4/$5$}
\draw(\x*0.2,-0.02) node[below] {\strut \xtext};
\coint{0}{0.2}{Red};
\coint{0.2}{0.6}{Cerulean};
\coint{0.6}{0.8}{Red};
\end{tikzpicture}
\end{equation*}
which is turned into the colouring \eqref{col:improved-3} in the proof of \Cref{thm:upper-bound}, possibly rearranging the colours.
The same happens for $r = 2$, where the fact that the real interval $[1,2)$ has no solution to $x + y = z$ is leveraged to the construction of \eqref{col:improved-2}.

For $r = 4$, we observe that $I(3) = S^*(3) = 14$.
In \Cref{lem:small-schur}, we noted that the only $3$-colourings of $\set{1,\dotsc,13}$ without monochromatic solutions to $x + y = z$ and $x + y + 1 = z$ are of the form $\rR\rB\rB\rR\rG\rG\bfx\rG\rG\rR\rB\rB\rR$, where $\bfx \in \set{\rR, \rG}$, up to permuting colours.
This leads to the $3$-colourings of the real interval $(1,14]$:
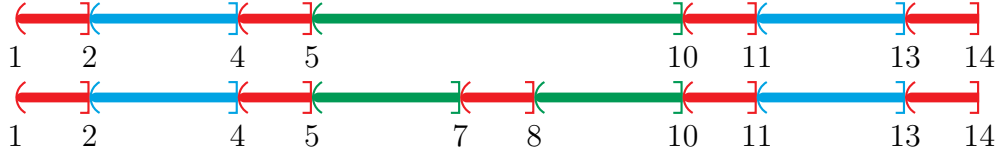
\begin{center}
\begin{tikzpicture}[scale=6.5]
\clip (-0.1,0.06) rectangle (1.95+0.1,-0.1);
\foreach \x/\xtext in
{0/$1$, 1/$2$, 3/$4$, 4/$5$, 9/$10$, 10/$11$, 12/$13$, 13/$14$}
\draw(\x*0.15,-0.02) node[below] {\strut \xtext};
\ocint{0.00}{0.15}{Red};
\ocint{0.15}{0.45}{Cerulean};
\ocint{0.45}{0.60}{Red};
\ocint{0.60}{1.35}{ForestGreen};
\ocint{1.35}{1.50}{Red};
\ocint{1.50}{1.80}{Cerulean};
\ocint{1.80}{1.95}{Red};
\end{tikzpicture}
\end{center}
\begin{center}
\begin{tikzpicture}[scale=6.5]
\clip (-0.1,0.04) rectangle (1.95+0.1,-0.1);
\foreach \x/\xtext in
{0/$1$, 1/$2$, 3/$4$, 4/$5$, 6/$7$, 7/$8$, 9/$10$, 10/$11$, 12/$13$, 13/$14$}
\draw(\x*0.15,-0.02) node[below] {\strut \xtext};
\ocint{0.00}{0.15}{Red};
\ocint{0.15}{0.45}{Cerulean};
\ocint{0.45}{0.60}{Red};
\ocint{0.60}{0.90}{ForestGreen};
\ocint{0.90}{1.05}{Red};
\ocint{1.05}{1.35}{ForestGreen};
\ocint{1.35}{1.50}{Red};
\ocint{1.50}{1.80}{Cerulean};
\ocint{1.80}{1.95}{Red};
\end{tikzpicture}
\end{center}
which, through \Cref{thm:upper-bound}, are respectively transformed into the $4$-colourings of $[2,N]$

\noindent\makebox[\textwidth]
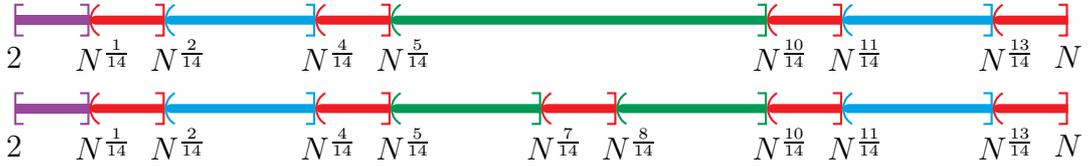
{%
\centering
\begin{tikzpicture}[scale=6.0]
\def\len{0.165}
\clip (-0.1,0.08) rectangle (14*\len+0.1,-0.15);
\foreach \x/\xtext in
{0/$2$,
1/$\;\;\;N^{\frac{1}{14}}$,
2/$\;\;\;N^{\frac{2}{14}}$,
4/$\;\;\;N^{\frac{4}{14}}$,
5/$\;\;\;N^{\frac{5}{14}}$,
10/$\;\;\;N^{\frac{10}{14}}$,
11/$\;\;\;N^{\frac{11}{14}}$,
13/$\;\;\;N^{\frac{13}{14}}$,
14/$N$}
\draw (\x*\len,-0.02) node[below] {\strut \xtext};
\ccint{0*\len}{1*\len}{Purple};
\ocint{1*\len}{2*\len}{Red};
\ocint{2*\len}{4*\len}{Cerulean};
\ocint{4*\len}{5*\len}{Red};
\ocint{5*\len}{10*\len}{ForestGreen};
\ocint{10*\len}{11*\len}{Red};
\ocint{11*\len}{13*\len}{Cerulean};
\ocint{13*\len}{14*\len}{Red};
\end{tikzpicture}
}
\noindent\makebox[\textwidth]{%
\centering
\begin{tikzpicture}[scale=6.0]
\def\len{0.165}
\clip (-0.1,0.04) rectangle (14*\len+0.1,-0.16);
\foreach \x/\xtext in
{0/$2$,
1/$\;\;\;N^{\frac{1}{14}}$,
2/$\;\;\;N^{\frac{2}{14}}$,
4/$\;\;\;N^{\frac{4}{14}}$,
5/$\;\;\;N^{\frac{5}{14}}$,
7/$\;\;\;N^{\frac{7}{14}}$,
8/$\;\;\;N^{\frac{8}{14}}$,
10/$\;\;\;N^{\frac{10}{14}}$,
11/$\;\;\;N^{\frac{11}{14}}$,
13/$\;\;\;N^{\frac{13}{14}}$,
14/$N$}
\draw (\x*\len,-0.02) node[below] {\strut \xtext};
\ccint{0*\len}{1*\len}{Purple};
\ocint{1*\len}{2*\len}{Red};
\ocint{2*\len}{4*\len}{Cerulean};
\ocint{4*\len}{5*\len}{Red};
\ocint{5*\len}{7*\len}{ForestGreen};
\ocint{7*\len}{8*\len}{Red};
\ocint{8*\len}{10*\len}{ForestGreen};
\ocint{10*\len}{11*\len}{Red};
\ocint{11*\len}{13*\len}{Cerulean};
\ocint{13*\len}{14*\len}{Red};
\end{tikzpicture}
}
Each of these colourings has the property that the only monochromatic solutions to $x y = z$ are contained in the initial interval $[2, (N/2)^{1/14}]$. Thus, there are
\begin{equation}
\label{eq:count-4}
    \p[\big]{1/(14 \cdot 2^{1/14})} N^{1/14} \log N + O\p[\big]{N^{1/14}}
\end{equation}
monochromatic solutions to $x y = z$ in total.
Curiously, the interval $((N/2)^{7/14}, N^{8/14}]$ can be coloured either green or red.
In fact, in the colourings above, this interval can be arbitrarily coloured in green and red without the creation of any additional monochromatic solution to $x y = z$.

While these $4$-colourings of $[2,N]$ have the right order of $N^{1/14}$ monochromatic solutions, we do not obtain a strong asymptotic result like \Cref{thm:two-colours} in this case.
If the count \eqref{eq:count-4} turns out to be asymptotically sharp, this would put a limitation on the possible forms of stability results.
Indeed, this would contradict an \emph{a priori} guess that the optimal colourings must have a strong interval structure.
Indeed, this indicates that this interval structure must accommodate the possibility that some of the intervals are arbitrarily coloured from a palette, rather than being prescribed a unique colour.

We do not obtain the correct exponent for $r = 5$ simply because we have not established whether $I(4) = S(4)$.
However, with a computer search, we were able to determine that $S^*(4) = 41$, and therefore $S^*(4) < S(4) = 45$.
Indeed, our search found all $576$ optimal colourings $c \from [40] \to \set{\rR, \rB, \rG, \rP}$ without monochromatic solutions to $x + y = z$ and $x + y + 1 = z$.
Up to permuting colours, these are all of the form
\begin{equation}
\label{eq:pattern-5}
\rR\rB\rB\rR\rG\rG\bfx\rG\rG\rR\rB\rB\rR
\rP\rP\bfy\rP\rP\bfz\bfw\bfs\bft\rP\rP\bfu\rP\rP
\rR\rB\rB\rR\rG\rG\bfv\rG\rG\rR\rB\rB\rR,
\end{equation}
where $\bfx, \bfv \in \set{\rR, \rG}$, $\bfy, \bfz, \bft, \bfu \in \set{\rR, \rP}$ and $\bfw, \bfs \in \set{\rB, \rG, \rP}$.
In any case, the fact that $S^*(4) < S(4)$ does not extinguish the possibility that $I(4) = S(4)$, which would then imply that \Cref{thm:lower-bound} has the optimal exponent for five colours as well.
The colourings described by \eqref{eq:pattern-5} can be converted into lower bounds for $I(4)$, and consequently for $\cN_4\p[\big]{[2,N], xy=z }$.
As an illustration, the simplest of the colourings in \eqref{eq:pattern-5} gives rise to the following $4$-colouring of $(1,41]$, the labels being omitted for clarity,
\begin{center}
\begin{tikzpicture}[scale=6.5]
\def\len{0.055}
\clip (-0.05,0.035) rectangle (40*\len+0.1,-0.035);
\ocint{0*\len}{1*\len}{Red};
\ocint{1*\len}{3*\len}{Cerulean};
\ocint{3*\len}{4*\len}{Red};
\ocint{4*\len}{9*\len}{ForestGreen};
\ocint{9*\len}{10*\len}{Red};
\ocint{10*\len}{12*\len}{Cerulean};
\ocint{12*\len}{13*\len}{Red};
\ocint{13*\len}{27*\len}{Purple};
\ocint{27*\len}{28*\len}{Red};
\ocint{28*\len}{30*\len}{Cerulean};
\ocint{30*\len}{31*\len}{Red};
\ocint{31*\len}{36*\len}{ForestGreen};
\ocint{36*\len}{37*\len}{Red};
\ocint{37*\len}{39*\len}{Cerulean};
\ocint{39*\len}{40*\len}{Red};
\end{tikzpicture}
\end{center}
Using the procedure outlined in the proof of \Cref{thm:upper-bound} and any of the colourings from \eqref{eq:pattern-5}, we obtain a $5$-colouring of $[2,N]$ with
\begin{equation}
\label{eq:count-5}
    \p[\big]{1/(41 \cdot 2^{1/41})} N^{1/41} \log N + O\p[\big]{N^{1/41}}
\end{equation}
monochromatic solutions to $x y = z$.
This time, we do not know whether the exponent $1/41$ is sharp, as \Cref{thm:lower-bound} only guarantees $c N^{1/45}$ monochromatic solutions.
Note that the discussion regarding stability and intervals being assigned a colour palette will be magnified if the count \eqref{eq:count-5} transpires to be asymptotically tight.

Due to the computational cost involved, we have not investigated the value of $S^*(5)$.
Indeed, Heule's proof~\cite{Heule2018-oe} that $S(5) = 161$ was a construction of a 2 petabyte long $\mathrm{SAT}$ certificate, which is currently the largest mathematical proof ever produced.
The computation of $S^*(5)$ is likely considerably easier than that of $S(5)$, as the added constraint could potentially reduce the computation by many orders of magnitude.


\section{Multiplicative equations}
\label{sec:multiplicative}

We now turn our attention to general multiplicative equations.
The purpose of this section is to prove \Cref{thm:lower-bound-a} and \Cref{thm:upper-bound-a}, which provide lower and upper bounds respectively for the number of monochromatic non-degenerate solutions to equations of the form
\begin{equation*}
    x_1^{a_1}x_2^{a_2} \dotsm x_k^{a_k} = y,
    \tag{\ref{eq:multiplicative}}
\end{equation*}
in any $r$-colouring of the interval $[2,N]$.
Let $\bfa = (a_1, \dotsc, a_k) \in \NN^k$ and, without loss of generality, assume that $a_1 \leq a_2 \leq \dotsb \leq a_k$.
We write $m(\bfa) = \card{\set {i \in k \st a_i = 1}}$ and we will always assume $m(\bfa) \geq 1$. These assumptions imply that $a_1=1$.

As with the multiplicative Schur equation $x y = z$, our bounds for the number of solutions for this equation are given in terms of Rado numbers associated with the corresponding additive equation
\begin{equation*}
    a_1 x_1 + a_2 x_2 + \dotsb + a_k x_k = y.
    \tag{\ref{eq:additive}}
\end{equation*}
Let $R_{\bfa}(r)$ denote the $r$-colour Rado number for this equation, namely, the minimum $N$ such that every $r$-colouring of $[N]$ has $x_1, \dotsc, x_k, y$ of the same colour that satisfy \eqref{eq:additive}.
The existence of these numbers for all $r$ follows from Rado's criterion and our assumption that $m(\bfa) \geq 1$.

\subsection{Lower bounds}

The proof of \Cref{thm:lower-bound-a} is an extension of the proof of \Cref{thm:lower-bound}.
Indeed, we use the same two-dimensional set
\begin{equation*}
    M_{b} \defined \set[\big]{ 2^{ji} b^i \st 1 \leq i \leq S,\; 1 \leq j \leq W } \cup \set[\big]{ 2^j \st 1 \leq j \leq W },
\end{equation*}
for some choice of $S$ and $W$.
As before, we write $M_{b,j}$ for the geometric progression
\begin{equation*}
    M_{b,j} \defined \set[\big]{ (2^j b)^i \st 1 \leq i \leq S },
\end{equation*}
and $M' = \set{2^j \st 1 \leq j \leq W}$, whence $M_b = \bigcup_{1 \leq j \leq W} M_{b,j} \cup M'$.

\begin{lemma}
\label{lem:ramsey-pattern-r-Gen}
Let $\bfa = (a_1, \dotsc, a_k)$ with $m(\bfa) \geq 1$.
For any $r \geq 2$, there exists $F(r,\bfa)$ such that if we set $S = R_{\bfa}(r-1)$ and $W = F(r,\bfa)$, then any $r$-colouring of the set $M_b$ contains a monochromatic non-degenerate solution to $x_1^{a_1} \dotsm x_k^{a_k} = y$ such that $y$ is a multiple of $b$.
\end{lemma}
\begin{proof}
As noted above, we may assume without loss of generality that $a_1=1$.
Let $A = a_2 + \dotsb + a_k$ and define $T = T(r)$ to be the smallest integer $T$ such that every $r$-colouring of $[T]$ contains $x_2, \dotsc, x_k, y$ of the same colour with $x_i \neq x_j$ for all $i \neq j$ such that
\begin{equation}
\label{eq:auxiliar}
    a_2 x_2 + \dotsb + a_k x_k = A y.
\end{equation}
The existence of such a $T$ may be obtained, for example, by appealing to the supersaturated version of Rado's theorem \cite[Theorem 1]{Frankl1988-zl}.
Since  all the $x_i$ are distinct, we have $T = T(r) \geq T(1) \geq k$.
Let $L = 1 + A \cdot T \cdot S!$ and take $F(r, \bfa)$ to be the van der Waerden number $W(L, r^S)$.

Given an $r$-colouring $c \from M_b \to [r]$, consider the auxiliary colouring $\xi \from [W] \to [r]^S$ obtained by colouring an integer $j$ by the colour pattern observed in $M_{b,j}$.
In other words, we define $\xi$ by
\begin{equation*}
    \xi(j) \defined \p[\big]{c(2^j b), c(2^{2j} b^2), \dotsc, c(2^{Sj}b^S)} \in [r]^S.
\end{equation*}
Now, since $W = W\p[\big]{L,r^S}$, there is an arithmetic progression
\begin{equation*}
    \set[\big]{j_0, j_0 + d, j_0 + 2d, \dotsc, j_0 + ATS!d} \subseteq [W]
\end{equation*}
that is monochromatic in the colouring $\xi$.
That is to say, for every $1 \leq i \leq S$, we have
\begin{equation*}
    c((2^{j_0}b)^i) = c((2^{j_0 + d}b)^i) = \dotsb = c((2^{j_0 + ATS!d}b)^i).
\end{equation*}
We distinguish two cases, depending on how many colours appears in $M_{b,j_0}$.

\begin{description}[labelindent=\parindent, leftmargin=2\parindent]
\item[First case] no more than $r-1$ colours are used in $M_{b,j_0}$.
Then, as $M_{b,j_0}$ is a geometric progression of length $R_\bfa(r-1)$, there is a monochromatic solution to $x_1^{a_1} \dotsm x_k^{a_k} = y$ in $M_{b,j_0}$, namely, for some $i_1, \dotsc, i_k, t \in [S]$, we have $x_s = (2^{j_0}b)^{i_s}$ and $y = (2^{j_0}b)^{t}$.
Since this solution may be degenerate if the $i_s$ are not all distinct, we need to modify this solution into a non-degenerate one.
Define $h_1, \dotsc, h_k$ inductively by setting $h_1 = 0$, $x_1'= x_1$ and choosing $h_n$ such that the number $x_n' = (2^{j_0 + d t h_n}b)^{i_n}$ is distinct from $x_1', \dotsc, x_{n-1}'$.
By the pigeonhole principle, we can always choose $h_n$ such that $0 \leq h_n \leq n \leq k$.
Then set $\mu = a_1 h_1 i_1 + \dotsb a_k h_k i_k$ and observe that $\mu \leq A k S \leq A T S! \leq L-1$, so we can set $y' = (2^{j_0 + \mu d} b)^t$.
Thus
\begin{align*}
    y' = (2^{j_0 + \mu d} b)^t
    &= y 2^{\mu d t}
    = \p[\big]{x_1^{a_1} \dotsm x_k^{a_k}} \p[\big]{ 2^{d a_1 h_i i_i} \dotsm 2^{d a_k h_k i_k} } \\
    &= \p[\big]{(2^{j_0}b)^{a_1 i_1} \dotsm (2^{j_0}b)^{a_k i_k}} \p[\big]{2^{dt a_1 h_i i_i} \dotsm 2^{dt a_k h_k i_k}} \\
    &= (2^{j_0 + dt h_1}b)^{a_1 i_1} \dotsm (2^{j_0 + dt h_k}b)^{a_k i_k}
    = (x_1')^{a_1} \dotsm (x_k')^{a_k}.
\end{align*}
As $c(x_s') = c(x_s)$ and $c(y') = c(y)$, we obtain a monochromatic non-degenerate solution to \eqref{eq:multiplicative} such that $b$ divides $y'$.

\item[Second case] all $r$ colours are available in $M_{b,j_0}$.
Consider the set
\begin{equation*}
    M'' = \set{2^{S d j} \st 1 \leq j \leq T}
\end{equation*}
and observe that $M'' \subseteq M'$ as $S! d T \leq W$.
Since $M''$ is a geometric progression of length $T$, we can therefore find a monochromatic solution to the multiplicative analogue of \eqref{eq:auxiliar}, namely,
\begin{equation*}
    x_2^{a_2} \dotsm x_k^{a_k} = y^A,
\end{equation*}
where $x_s = 2^{S! d u_s}$ for $2 \leq s \leq k$ and $y = 2^{S! d v}$ for distinct $u_2, \dotsc, u_k, v \in [T]$.

As every colour is available in $M_{b, j_0}$, there is $1 \leq i \leq S$ where the colour of $(2^{j_0}b)^i$ is the same as the colour of $x_2, \dotsc, x_k, y$ as found above.
Let $x_1 = (2^{j_0}b)^i$ and $y' = (2^{j_0 + (S!/i) v Ad} b)^i$. Now observe that
\begin{equation*}
    c(x_1) = c(x_2) = \dotsb = c(x_k) = c(y'),
\end{equation*}
and that all of $x_1, \dotsc, x_k, y'$ are distinct.
Moreover,
\begin{align*}
    y' &= (2^{j_0 + (S!/i) v Ad} b)^i
    = (2^{j_0} b)^i 2^{S! vAd} \\
    &= x_1 (2^{S!dv})^{A} = x_1 y^A
    = x_1^{a_1} x_2^{a_2} \dotsm x_k^{a_k},
\end{align*}
where in the last equality we made use of the assumption $a_1 = 1$.
Finally, note that $b$ divides $y'$ as required.
\end{description}

In any case, we find a monochromatic non-degenerate solution to \eqref{eq:multiplicative} with $b \divides y$.
\end{proof}

The proof of \Cref{thm:lower-bound-a} is now quite similar to the proof of \Cref{thm:lower-bound}.
We delay the proof that the exponent $1/R_\bfa(r-1)$ is sharp until the next subsection, where it is established by \Cref{prop:rado-two-colour}.

\begin{proof}[Proof of Theorem \ref{thm:lower-bound-a}]
Choose $S = R_\bfa(r-1)$ and $W = F(r, \bfa)$ as given by \Cref{lem:ramsey-pattern-r-Gen}.
We are going to find different copies of the set $M_b$ inside $[2,N]$ each of which provides a distinct monochromatic non-degenerate solution to $x_1^{a_1} \dotsm x_k^{a_k} = y$.
To guarantee that $M_b \subseteq [2,N]$, we must choose $2 \leq b \leq N^{1/S}2^{-W}$.
We restrict ourselves to $b$ in the set
\begin{equation*}
    \cB \defined \set[\Big]{ 2 \leq b \leq \floor{N^{1/S}2^{-W}} \st 2 \ndivides b, \text{ $b$ is not a perfect power } }.
\end{equation*}
Each solution obtained by an application of \Cref{lem:ramsey-pattern-r-Gen} is guaranteed to have $b \divides y$, that is, $y = 2^{ji} b^i$ for some $1 \leq i \leq S$ and $1 \leq j \leq W$.
Given this value of $y$, we can uniquely retrieve the value of $b$ by repeatedly dividing $y$ by $2$ and then taking the smallest integer $k$-th root, exactly as we did in the proof of \Cref{thm:lower-bound}.
Therefore, for each $b \in \cB$ we find a different monochromatic non-degenerate solution to \eqref{eq:multiplicative}.

Since only a negligible proportion of integers less than $N^{1/S}$ are perfect powers, we have $\card{\cB} \geq \p[\big]{2^{-W-1} - o(1)} N^{1/S}$ and we are done.
\end{proof}


\subsection{Upper bounds}

Given $\bfa = (a_1, \dotsc, a_k) \in \NN^k$, recall the additive equation
\begin{equation*}
    a_1 x_1 + a_2 x_2 + \dotsb + a_k x_k = y.
    \tag{\ref{eq:additive}}
\end{equation*}
In analogy with the interval Schur numbers $I(r)$ considered in \S\ref{sec:constructions}, we introduce the corresponding interval Rado numbers
\begin{equation*}
    I_\bfa(r) \defined \inf \set[\Bigg]{ T \in \RR \st
    \begin{array}{c}
    \text{for any $r$-colouring of the real interval $[1,T]$ there are} \\
    \text{$x_1, \dotsc, x_k, y \in [1,T]$ of the same colour satisfying \eqref{eq:additive}}
    \end{array}
    }.
\end{equation*}
We immediately see that
\begin{equation*}
    1\leq I_{\bfa}(r) \leq R_{\bfa}(r)
\end{equation*}
holds for all $r \geq 1$, where $R_{\bfa}(r) \in \NN$ are the Schur numbers
\begin{equation*}
    R_\bfa(r) \defined \min \set[\Bigg]{ N \in \NN \st
    \begin{array}{c}
    \text{for any $r$-colouring of the integer interval $[N]$ there are} \\
    \text{$x_1, \dotsc, x_k, y \in [N]$ of the same colour satisfying \eqref{eq:additive}}
    \end{array}
    }.
\end{equation*}
In particular, the interval Schur numbers $I_{\bfa}(r)$ are finite whenever $m(\bfa) \geq 1$.

By a compactness argument, we can define $I_\bfa(r)$ alternatively as the maximum $T$ such that there is an $r$-colouring of the real interval $(1, T]$ without any monochromatic solutions to \eqref{eq:additive}.
The proof is identical to that of \Cref{lem:compact}, so it is omitted.

\begin{lemma}
\label{lem:compact-a}
Let $\bfa = (a_1, \dotsc, a_k)$ be such that $m(\bfa) \geq 1$.
Then for all $r \geq 1$ there is an $r$-colouring of the half-open real interval $(1, I_\bfa(r)]$ with no monochromatic solution to \eqref{eq:additive}.
Moreover, for all $T > I_\bfa(r)$, no such $r$-colouring exists for $(1,T]$.
\qed
\end{lemma}

Before proceeding to the proof of \Cref{thm:upper-bound-a}, we first record an asymptotic estimate for the number of solutions to \eqref{eq:multiplicative}.

\begin{proposition}
\label{prop:solution-count}
Let $\bfa = (a_1, \dotsc, a_k) \in \NN^k$ be such that $1 = a_1 \leq \dotsb \leq a_k$ and $m = m(\bfa) \geq 1$ be maximal such that $a_m = 1$.
Then there is a constant $C(\bfa) > 0$ such that, as $X \to \infty$, the number of solutions to the equation $x_1^{a_1} \dotsm x_k^{a_k} = y$ in the integer interval $[X]$ satisfies
\begin{equation*}
    \cN_1\p[\big]{ x_1^{a_1} \dotsm x_k^{a_k} = y, [X]}
    = C(\bfa) X (\log X)^{m - 1} + O_\bfa\p[\big]{ X(\log X)^{m-2}}.
\end{equation*}
Moreover, the leading constant is given by
\begin{equation*}
    C(\bfa) = \frac{\zeta(a_{m+1}) \dotsm \zeta(a_k)}{(m - 1)!},
\end{equation*}
where $\zeta$ is the Riemann zeta function.
\end{proposition}
\begin{proof}
We induct on the value of $k - m$.
The case $m = k$ is classical (see for instance the book of Montgomery and Vaughan~\cite[\S2.1.1, Exercise 18]{Montgomery2007-ts}).
Indeed, we have
\begin{equation*}
    \cN_1\p[\big]{ x_1 \dotsm x_k = y, [X]}
    = X P_k(\log X) + O_k \p[\big]{X^{1 - 1/k}(\log X)^{k - 2}},
\end{equation*}
where $P_k$ is a real polynomial of degree $k-1$ and leading coefficient $1/(k-1)!$.

Now suppose that $m < k$, whence $a_k \geq 2$.
Assume the induction hypothesis that
\begin{equation*}
    F(X) \defined
    \cN_1\p[\big]{ x_1^{a_1} \dotsm x_{k-1}^{a_{k-1}} = y, [X]}
    = C X (\log X)^{m-1} + O_{(a_1, \dotsc, a_{k-1})}\p[\big]{X (\log X)^{m-2}},
\end{equation*}
as $X \to \infty$, where $C = C(a_1, \dotsc, a_{k-1}) = \zeta(a_{m+1}) \dotsm \zeta(a_{k-1}) / (m - 1)! > 0$.
Using the fact that $\sum_{n=1}^{\infty}n^{-t}(\log n)^h$ absolutely converges for all $h\geq 0$ and $t>1$, the binomial theorem implies that
\begin{equation*}
    \sum_{x \leq X^{1/a_k}}\frac{X \p[\big]{\log (X/x^{a_k})}^{m-1}}{x^{a_k}}
    = \p[\bigg]{ \sum_{x \leq X^{1/a_k}} \frac{1}{x^{a_k}} }X(\log X)^{m-1} + O_{a_k,m}\p[\big]{ X(\log X)^{m-2}}.
\end{equation*}
Now note that the number of tuples $(x_1, \dotsc, x_k, y) \in [X]^{k+1}$ with $x_1^{a_1} \dotsb x_{k}^{a_k} = y$ and $x_k = x$ is given by $F(X/x^{a_k})$.
In particular, such a tuple can only exist if $x \leq X^{1/a_k}$.
Thus, by the induction hypothesis, we have that $\cN_1\p[\big]{ x_1^{a_1} \dotsm x_{k}^{a_{k}} = y, [X]}$ is equal to
\begin{align*}
    \sum_{x \leq X^{1/a_k}}\!\!\!\! F\p[\big]{X/x^{a_k} }
    &= \sum_{x \leq X^{1/a_k}} \p[\bigg]{ \frac{C X \p[\big]{\log\p{X/x^{a_k}}}^{m-1}}{x^{a_k}} + O_{(a_1, \dotsc, a_{k-1})}\p[\Big]{\frac{X (\log X)^{m-2}}{x^{a_k}}} }\\
    &= \p[\bigg]{ \sum_{x \leq X^{1/a_k}} \frac{1}{x^{a_k}} } C X (\log X)^{m-1}
    + O_{(a_1, \dotsc, a_k)}\p[\big]{X (\log X)^{m-2}}.
\end{align*}
Since $a_k \geq 2$, we have the tail estimate $\sum_{x > Y} x^{-a_k} = O(Y^{1 - a_k})$ as $Y \to \infty$.
Therefore, we get $C(a_1, \dotsc, a_k) = \zeta(a_k) C(a_1, \dotsc, a_{k-1})$, and we are done.
\end{proof}

We are now ready to prove \Cref{thm:upper-bound-a} by following the same strategy as in our proof of \Cref{thm:upper-bound}.

\begin{proof}[Proof of \Cref{thm:upper-bound-a}]
By \Cref{lem:compact-a}, there is an $(r-1)$-colouring $\xi$ of the real interval $(1, I_\bfa(r-1)]$ without monochromatic solutions to \eqref{eq:multiplicative}.
We construct an $r$-colouring of the integer interval $[2,N]$ in the following way.
For $M = N^{1/I_\bfa(r-1)}$, we colour the interval $[2,M]$ of the new colour $r$ and use $\xi$ to fill the interval $(M,N]$ as follows:
\begin{equation*}
    c(x) \defined \begin{cases}
        r & \text{if $2 \leq x \leq M$} \\
        \xi\p[\big]{\log x / \log M} & \text{if $M < x \leq N$.}
    \end{cases}
\end{equation*}
Since every $x \in (M, N]$ satisfies $\log x / \log M \in (1, I_\bfa(r-1)]$, this defines a colouring of the entire real interval $[2,N]$.

We claim that the colouring $c$ has no monochromatic solutions to \eqref{eq:multiplicative} with colours lying in $\set{1, \dotsc, r - 1}$.
Indeed, if we had such a solution, then $M < x_1, \dotsc, x_k, y \leq N/2$ and $x_i = M^{x_i'}$, $y = M^{y'}$ for some $x_i', y' \in (1, I_\bfa(r-1)]$.
But $c(x) = \xi(x_i') = c(y) = \xi(y')$ for all $i$, and \eqref{eq:multiplicative} implies that $a_1 x_1' + \dotsb a_k x_k' = y'$, contradicting our choice of $\xi$.

Therefore, all monochromatic solutions are in colour $r$ and must come from the interval $[2,M]$ alone.
By \Cref{prop:solution-count}, there are
\begin{equation*}
    C(\bfa) M (\log M)^{m(\bfa)-1} + O\p[\big]{ M (\log M)^{m(\bfa) - 2}}
\end{equation*}
solutions to \eqref{eq:multiplicative} in this interval.
The leading term gives
\begin{align*}
    C(\bfa) M (\log M)^{m(\bfa)-1}
    &= C_{r,\bfa} N^{1/I(r-1)} (\log N)^{m(\bfa) - 1} + O\p[\big]{N^{1/I(r-1)} (\log N)^{m(\bfa)-2}},
\end{align*}
where $C_{r,\bfa} = C(\bfa)/(I_\bfa(r-1))^{m(\bfa)-1}$, and $C(\bfa)$ is given by \Cref{prop:solution-count}.
\end{proof}

To compare the bounds given by \Cref{thm:lower-bound-a} and \Cref{thm:upper-bound-a}, we need to understand the relationship between the Schur numbers $R_\bfa(r)$ and $I_\bfa(r)$.
For arbitrary $\bfa$, we show that these quantities coincide when $r \leq 2$.
In particular, the following result demonstrates that the exponent of $1/R_\bfa(r - 1)$ in \Cref{thm:lower-bound-a} is sharp for all $r \leq 3$.

\begin{proposition}
\label{prop:rado-two-colour}
Let $\bfa = (a_1, \dotsc, a_k) \in \NN^{k}$ and $A = a_1 + \dotsb + a_k$.
If $a_1 = 1$, then
\begin{align*}
    I_{\bfa}(1) = R_{\bfa}(1) &= A, \\
    I_{\bfa}(2) = R_{\bfa}(2) &= A^2 + A - 1.
\end{align*}
\end{proposition}
\begin{proof}
By averaging, if $x_1, \dotsc, x_k, y \geq 1$ satisfy \eqref{eq:additive}, then
\begin{equation}
\label{eq:average}
    \min_{i \in [k]} x_i \leq \frac{y}{A} \leq \max_{i \in [k]} x_i.
\end{equation}
It follows that \eqref{eq:additive} admits solutions over the real interval $[1,T]$ if and only if $T \geq A$.
This gives the first equality.

We now address the second equality.
It was determined by Funar~\cite{Funar1990-is} that $R_{\bfa}(2) = A^2 + A - 1$.
Thus, it suffices to find a $2$-colouring of the real interval $[1, R_{\bfa}(2))$ with no monochromatic solutions to \eqref{eq:additive}.
Consider the colouring $[1, A^2 + A - 1) = R \sqcup B$ with
\begin{equation*}
    R \defined [1, A) \sqcup [A^2, A^2 + A - 1), \quad B \defined [A, A^2).
\end{equation*}
Suppose for a contradiction that there is a monochromatic solution $(x_1, \dotsc, x_k, y)$ to \eqref{eq:additive} in this colouring.
The lower bound in \eqref{eq:average} implies that this solution cannot come from the set $B$, so we must have $x_1, \dotsc, x_k, y \in R$.
As there are no solutions to \eqref{eq:additive} over $[1, A)$, we have $y \geq A^2$.
Hence, by the upper bound in \eqref{eq:average} and the definition of $R$, there exists $j \in [k]$ such that $x_j \geq A^2$.
Therefore,
\begin{equation*}
    (a_1 x_1 + \dotsb + a_k x_k) - a_j x_j = y - a_j x_j < A^2 + A - 1 - a_j A^2 \leq A - a_j,
\end{equation*}
which is absurd, as it implies $x_i < 1$ for some $i \in [k] \setminus \set{j}$.
\end{proof}


\section{Remarks and further directions}
\label{sec:questions}

With \Cref{thm:lower-bound} and \Cref{thm:upper-bound} combined, we determined the order of growth of $\cN_r\p[\big]{x y = z, [2,N]}$, up to a logarithmic factor, for $r \in \set{2, 3, 4}$.
The most pressing question is then to determine the order of growth for $r \geq 5$.
\begin{question}
Up to to polylogarithmitc factors, what is the minimum number of monochromatic solutions to $x y = z$ in an $r$-colouring of $\set{2, \dotsc, N}$, when $r \geq 5$?
\end{question}

In \Cref{thm:two-colours}, we gave a precise answer for the number of monochromatic solutions to $xy=z$ for two colours. It would be interesting if these methods could be applied to the three colour case.
\begin{question}
Is it true that every $3$-colouring of $[2,N]$ has at least
\begin{equation*}
    \p[\bigg]{\frac{1}{5 \cdot 2^{4/5}} - o(1)}N^{1/5} \log N
\end{equation*}
monochromatic solutions to $x y = z$ as $N \to \infty$?
\end{question}

If this is true, then the $3$-colouring of $[2,N]$
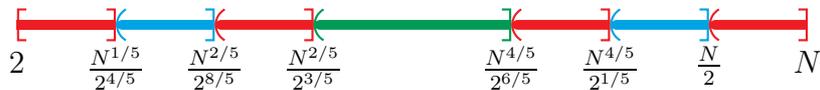
\begin{equation*}
\begin{tikzpicture}[baseline=(current  bounding  box.center), scale=6.5]
\clip (-0.1,0.04) rectangle (1.6+0.1,-0.15);
\foreach \x/\xtext in
{0/$2$, 1/$\frac{N^{1/5}}{2^{4/5}}$, 2/$\frac{N^{2/5}}{2^{8/5}}$, 3/$\frac{N^{2/5}}{2^{3/5}}$, 5/$\frac{N^{4/5}}{2^{6/5}}$, 6/$\frac{N^{4/5}}{2^{1/5}}$,7/$\frac{N}{2}$,8/$N$}
\draw (\x*0.2,-0.02) node[below] {\strut \xtext};
\ccint{0}{0.2}{Red};
\ocint{0.2}{0.4}{Cerulean};
\ocint{0.4}{0.6}{Red};
\ocint{0.6}{1.0}{ForestGreen};
\ocint{1.0}{1.2}{Red};
\ocint{1.2}{1.4}{Cerulean};
\ocint{1.4}{1.6}{Red};
\end{tikzpicture}
\end{equation*}
asymptotically minimises the number of monochromatic solutions to $x y = z$.


\subsection{Schur-type numbers}

It was observed by Schur~\cite{Schur1916-lc} that $S(r+1) \geq 3S(r) - 1$ for all $r \geq 1$.
Iterating this inequality gives $S(r) \geq (3^r + 1)/2$.
This bound is far from sharp, since Exoo~\cite{Exoo1994-tu} proved that $S(r) \geq c(321)^{r/5} > c(3.17176)^r$ for some $c > 0$.
Schur himself established the super-exponential upper bound $S(r) \leq \floor{ e\, r! }$.
It remains an open question whether an exponential upper bound holds for $S(r)$.
In their work, Abbott and Hanson~\cite{Abbott1972-sm} determined the following relation between $S(r)$ and $S^*(r)$.

\begin{theorem}[{\cite[Theorem 4.1]{Abbott1972-sm}}]
\label{thm:abott-hanson}
For every $r, t \geq 1$, we have
\begin{equation*}
    S^*(r + t) \geq 2S(r)S^*(t)  - S(r) - S^*(t) + 1.
\end{equation*}
\end{theorem}

In particular, we have that $S^*(r+1) \geq 3S(r) - 1$.
This implies that even with the additional constraint, $S^*(r)$ cannot grow significantly slower\footnote{In a previous version of this manuscript, we had a proof of $S^*(r+1) \geq 3S^*(r) - 1$ and a proof of $I(r+1) \geq 3S(r) - 1$.
We thank Paul Balister for noticing that our proof gave $S^*(r+1) \geq 3S(r)-1$.
In fact, this is a consequence of \Cref{thm:abott-hanson} by Abbott and Hanson.
In particular, we do not have $S^*(r+1) = 3S^*(r) - 1$ for all $r \geq 1$, which negatively answers a question we had previously posed.}
 than $S(r)$.

\subsection{Product Schur numbers}

One could define the analogue of Schur numbers for the product equation $x y = z$.
Namely, let $P(r)$ be the minimum $N$ such that every $r$-colouring of $\set{2,\dotsc,N}$ has a nontrivial monochromatic solution to $x y = z$.
One quickly notices that this is not a very interesting quantity as in fact we have $P(r) = 2^{S(r)}$.
Indeed, if $N = 2^{S(r)}$, then we can find a solution to $x y = z$ in the set $\set{2, 2^2, \dotsc, 2^{S(r)}}$.
However, if $N < 2^{S(r)}$ then consider the colouring $c \from [N] \to [r]$ defined by $c(x) = \chi(\Omega(x))$, where $\Omega(x)$ is the number of prime factors of $x$ with multiplicity and $\xi \from [S(r)-1] \to [r]$ is a colouring without monochromatic solutions to $x + y = z$.
The same applies for other multiplicative equations.

\subsection{Interval Schur numbers}

Consider now the quantity $I(r)$.
\Cref{lem:compact} informs us that for all real $T>1$, the half-open interval $[1,T)$ admits a sum-free $r$-colouring if and only if $T\leq I(r)$.
Moreover, for any such $T$, a sum-free $r$-colouring of $[1,T)$ is given by taking the restriction of a sum-free $r$-colouring of $[1,I(r))$ (which exists by \Cref{lem:compact}).
This is analogous to the integer setting: a sum-free $r$-colouring of $[N]$ exists if and only if $N<S(r)$, and, in the latter case, such a colouring can be found by restricting a suitable $r$-colouring of $[S(r)-1]$.

In spite of these similarities, note that it is not clear a priori from the definition of $I(r)$ that it must be an integer, or even a rational number.
Nevertheless, we believe that the bounds given by \Cref{thm:lower-bound} and \Cref{thm:upper-bound} should match up to logarithmic factors.
This claim would follow immediately from the following conjecture.

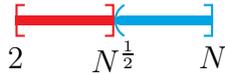
\begin{conjecture}
For all $r$, we have $I(r) = S(r)$.
\end{conjecture}


\subsection{Stability}

In \Cref{thm:two-colours}, we established that in a $2$-colouring of $[2, N]$, the number of monochromatic solutions to $x y = z$ is asymptotically minimised by the colouring $[2, N] = R \sqcup B$ with $R = [2, M^{1/2}] \sqcup (M, N]$ and $B = (M^{1/2}, M/2]$, where $M = N/2$.
That is, the following colouring
\begin{center}
\begin{tikzpicture}[baseline=(current  bounding  box.center), scale=6.5]
\clip (-0.1,0.04) rectangle (0.6+0.1,-0.13);
\foreach \x/\xtext in
{0/$2$, 1/$(\frac{N}{2})^{\frac{1}{2}}$, 2/$\phantom{((}\frac{N}{2}\phantom{)^{\frac{1}{2}}}$, 3/$N$}
\draw (\x*0.2,-0.02) node[below] {\strut \xtext};
\ccint{0}{0.2}{Red};
\ocint{0.2}{0.4}{Cerulean};
\ocint{0.4}{0.6}{Red};
\end{tikzpicture}
\end{center}
While this colouring is roughly balanced, namely, both colours occur in roughly $N/2$ elements, this is far from necessary.
Indeed, let $A = ( (N/2)^{1/2}, (2N)^{1/2}]$ and observe that if an element in $(N/2, N]$ can be written as a product of two blue elements in the colouring above, then both factors of the product must be in $A$.
Since $(N/2, N]$ is product-free, this means that if an element in $(N/2, N]$ is not in the product set $A \cdot A$, then it can be turned blue without creating any additional monochromatic solutions.
However, $\card{A \cdot A} \leq o(N)$ by the multiplication table problem\footnote{Very precise estimates are available after the work of Ford~\cite{Ford2008-en}.}.
Therefore, almost every element in $(N/2, N]$ can be turned blue.
Of course, almost every element was already blue in the original colouring
\begin{center}
\begin{tikzpicture}[baseline=(current  bounding  box.center), scale=6.5]
\clip (-0.1,0.04) rectangle (0.4+0.1,-0.11);
\foreach \x/\xtext in
{0/$2$, 1/$N^{\frac{1}{2}}$, 2/$N$}
\draw (\x*0.2,-0.02) node[below] {\strut \xtext};
\ccint{0}{0.2}{Red};
\ocint{0.2}{0.4}{Cerulean};
\end{tikzpicture}
\end{center}
but this does not minimise the number of monochromatic solutions to $x y = z$.

As we just saw, there are some serious restrictions on how much structural information can be obtained from a stability statement for colourings with few monochromatic solutions to $x y = z$, even in two colours. For more colours, we discussed in \S\ref{sec:constructions} that even our templates sometimes allow for the possibility of a whole interval in the logarithmic scale to be arbitrarily coloured from some palette of colours without creating a single monochromatic solution.
This is compounded by the considerations above, which essentially stem from the fact that integer intervals satisfy $\card[\big]{[N] \cdot [N]} = o(N^2)$, whilst for real intervals we have $[1,N] \cdot [1,N] = [1, N^2]$.
Nevertheless, we expect that a stability result similar to \Cref{thm:stability} should hold for more than two colours.


\subsection{General interval Rado numbers}

As with the Schur-type numbers, it is natural to speculate on the relationship between the quantities $I_\bfa(r)$ and $R_\bfa(r)$ for arbitrary $r$ and $\bfa$.
If, for a given $\bfa$, we had $I_\bfa(r)=R_\bfa(r)$ for all $r$, then our lower and upper bounds (\Cref{thm:lower-bound-a} and \Cref{thm:upper-bound-a} respectively) for the number of monochromatic solutions to \eqref{eq:multiplicative} would match up to polylogarithmic factors.
We showed in \Cref{prop:rado-two-colour} that this is true for $r \in \set{1, 2}$.
We therefore conclude by posing the following question.

\begin{question}
Given $k \geq 2$ and $\bfa = (a_1, \dotsc, a_k) \in \NN^k$, for which $r \geq 3$ is it true that $I_\bfa(r) = R_\bfa(r)$?
\end{question}


\section*{Acknowledgements}

We are grateful to Sean Prendiville for asking such an interesting question at the 2022 British Combinatorial Conference and for introducing J. Chapman to L. Aragão, M. Ortega and V. Souza.
We thank Rob Morris for helpful comments on an earlier draft of this paper and
Julian Sahasrabudhe for helpful comments on a late draft of this paper.
L. Aragão was supported by the Coordenação de Aperfeiçoamento de Pessoal de Nível Superior - Brasil (CAPES) - Finance Code 001. J. Chapman is supported by the Heilbronn Institute for Mathematical Research.


\printbibliography


\end{document}